\documentclass[reqno, 12pt]{amsart}
\usepackage{graphicx} 
\usepackage[utf8]{inputenc}
\usepackage[british]{babel}
\usepackage[OT2, T1]{fontenc}
\usepackage{lmodern}
\usepackage{amsmath,amssymb,amsthm,todonotes,booktabs,graphicx,longtable, bbm, nicefrac, xfrac, adjustbox,color}
\usepackage{enumitem}
\usepackage{appendix}
\usepackage{tikz-cd}
\usepackage{url, hyperref}
\usepackage[margin=1.2in, hmarginratio=1:1]{geometry}
\definecolor{dblue}{rgb}{0,0,0.7}
\definecolor{dartmouthgreen}{rgb}{0.05, 0.5, 0.06}   
\definecolor{dlilac}{rgb}{0.6, 0.33, 0.73}
\definecolor{ffuchsia}{rgb}{0.96, 0.0, 0.63}
\definecolor{hpurple}{rgb}{0.32, 0.09, 0.98}
\newtheorem{thm}{Theorem}
\newtheorem{prop}[thm]{Proposition}
\newtheorem{cor}[thm]{Corollary}
\newtheorem{lem}[thm]{Lemma}

\numberwithin{thm}{section}
\numberwithin{equation}{section}

\theoremstyle{definition}
\newtheorem{defi}[thm]{Definition}

\newtheorem{nota}[thm]{Notation}

\newcommand{\FF}{\mathbb{F}}
\newcommand{\LL}{\mathbb{L}}
\newcommand{\QQ}{\mathbb{Q}}
\newcommand{\ZZ}{\mathbb{Z}}
\newcommand{\RR}{\mathbb{R}}
\newcommand{\CC}{\mathbb{C}}

\newcommand{\NN}{\mathbb{N}}
\newcommand{\AAA}{\mathbb{A}}

\newcommand{\ZlZ}{\ZZ/ \ell\ZZ}
\newcommand{\ZmZ}{\ZZ/m\ZZ}
\newcommand{\ZpZ}{\ZZ/p\ZZ}

\newcommand{\hatM}{\hat{\mathcal{M}}'}

\DeclareMathOperator{\conj}{Conj}

\DeclareMathOperator{\tame}{tame}
\DeclareMathOperator{\wild}{wild}
\DeclareMathOperator{\Hom}{Hom}
\DeclareMathOperator{\aut}{Aut}

\DeclareMathOperator{\Gal}{Gal}

\DeclareMathOperator{\topo}{top}

\DeclareMathOperator{\unr}{nr}

\DeclareMathOperator{\age}{age}

\DeclareMathOperator{\leng}{length}
\DeclareMathOperator{\Frob}{Frob}

\DeclareMathOperator{\id}{id}
\DeclareMathOperator{\spec}{Spec}
\DeclareMathOperator{\str}{st}
\DeclareMathOperator{\Var}{Var}
\DeclareMathOperator{\rad}{rad}
\DeclareMathOperator{\et}{-\'Et}

\DeclareSymbolFont{cyrletters}{OT2}{wncyr}{m}{n}
\DeclareMathSymbol{\Sha}{\mathalpha}{cyrletters}{"58}
\vspace{10cm}
\vspace{50cm}
\vfill
\newpage

\title[The McKay correspondence for wild-by-tame split metacyclic groups]{The McKay correspondence and local heights for wild-by-tame split metacyclic groups}
\author{Julie Tavernier}
\date{March 2026}
\address{Department of Mathematical Sciences, University of Bath,
Claverton Down, Bath, BA2 7AY, UK}
\email{jlt86@bath.ac.uk}

\author{Takehiko Yasuda}
\address{Department of Mathematics, Graduate School of Science, Osaka University, Toyonaka, Osaka 560-0043, Japan}
\address{Kavli Institute for the Physics and Mathematics of the Universe, The University of Tokyo, 5-1-5 Kashiwanoha, Kashiwa, Chiba, 277-8583, Japan}
\email{yasuda.takehiko.sci@osaka-u.ac.jp}
\begin{document}
\begin{abstract}
    We study the McKay correspondence for the representations of certain wild-by-tame split metacyclic groups whose order is divisible by the characteristic of the base field.
    We calculate the stringy motive of the quotient variety and find a formula for its stringy Euler number.
    As a consequence, we prove that a crepant resolution of the quotient variety (provided one exists) does not in general have Euler characteristic equal to the number of conjugacy classes in $G$, in contrast to the classical case. In particular, we show it depends on the choice of representation as well as the group.
   As part of this, we compute the $\mathbf{v}$-function associated to a $G$-representation, corresponding to a stacky local height function.
\end{abstract}
\maketitle
\tableofcontents
\section{Introduction}

\noindent In \cite{Batyrev1999NonArchimedean}, Batyrev proved that if $G \subset SL_d(\CC)$ is a finite subgroup and there is a crepant resolution $Y \rightarrow \AAA_{\CC}^d/G$, then we have \[e(Y) = \#\conj(G)\] where $e(Y)$ is the topological Euler characteristic of $Y$ and $\#\conj(G)$ is the number of conjugacy classes of $G$.
This is a version of the McKay correspondence.
A more conceptual version of this result was later proved by Denef and Loeser in \cite{DenefLoeser2002}. 
This correspondence can be extended to fields of prime characteristic not dividing the order of the group, known as the "tame" case (see \cite{yasuda2006motivic} and \cite[$\S5.2$]{WoodYasuda2015MassI}), where we interpret $e(Y)$ to be the $\ell$-adic Euler characteristic.
The wild Mckay correspondence is a generalisation of this to when the characteristic of the base field divides $\# G$. This was proved by Yasuda in \cite{Yasuda2024WildDM}.
In this paper, we consider the wild McKay correspondence for groups of the form $G = \ZpZ \rtimes \ZmZ$, where $p$ is a prime and $m$ is an integer coprime to $p$.
Specifically, we prove a formula for the stringy motive $M_{\str}(\AAA_{k}^d/G)$ for groups of this form over an algebraically closed field $k$ of characteristic $p$.
We use this to find a formula for the stringy Euler number of $\AAA_{k}^d/G$, and find a criterion for the type of singularities that these quotient varieties can have.

\subsection{$\mathbf{v}$-functions}
A necessary step in the calculation of the wild McKay correspondence is the computation of the $\mathbf{v}$-function associated to a representation of the group.
If the order of the group is not divisible by the characteristic of the base field, the $\mathbf{v}$-function corresponds to the age function. 
However in the wild case, it is in general very difficult to compute, and has only been done in a few special cases.
The $\mathbf{v}$-function is additive with respect to direct sums, and so given any $G$-representation $V$, we may reduce to the indecomposable case. 
We state a special case of Theorem \ref{defn of v function} here. 
The full theorem may be found in Section \ref{subsec; v fn non connected}.
\begin{thm}[Theorem \ref{defn of v function}]
    Let $G= \ZpZ \rtimes \ZmZ$ and $L/k((t))$ be a field extension with Galois group $G$. 
    For $\gamma \in (\ZmZ)^{\times}$, let $I_{\gamma,r,s}$ be a subset of $\{0,1,\cdots, m-1\} \times \{0,1,\cdots, d-1\}$ of $m$ elements determined by a congruence condition modulo $m$. 
    Let $\Delta_{G,\gamma}$ denote the space of $G$-torsors corresponding to the $G$-extensions.
    Then the $\mathbf{v}$-function associated to an indecomposable $d$-dimensional representation $V_{d}$ of $G$ is the function $\mathbf{v}_{V_{d, \gamma}}:\Delta_{G,\gamma} \rightarrow \frac{1}{m}\ZZ$  given by \[\mathbf{v}_{V_{d, \gamma}}(x) = \begin{cases}
        0 & x \: \text{is trivial} \\
        \sum_{(i,j) \in I_{\gamma,r,s}}\frac{i}{m} + \sum_{(i,j) \in I_{\gamma,r,s}} 
        \left\lceil \frac{-ip - jr}{mp} \right\rceil & \text{otherwise},
    \end{cases}\] where $r$ is the last ramification jump of $L/k((t))$ and $g_x \in G$. 
    If $V = \bigoplus_{\lambda = 1}^lV_{d_{\lambda}}$ is a decomposition of $V$ into $l$ indecomposable $d_{\lambda}$-dimensional representations,
    then $\mathbf{v}_{V, \gamma} = \sum_{\lambda = 1}^l\mathbf{v}_{V_{d_{\lambda}, \gamma}}$.
\end{thm} 
\noindent Previous work on calculating $\mathbf{v}$-functions in the wild case include $\mathbf{v}$-functions of permutation representations \cite[Thm.~$4.8$]{WoodYasuda2015MassI}, those of stable hyperplanes in permutation representations \cite[Cor.~$11.4$]{Yasuda2016WilderMcKay}, the $\mathbf{v}$-functions associated to representations of $p$-cyclic groups \cite[Prop.~$6.9$]{Yasuda2014pCyclicMcKay} and $p^r$-cyclic groups \cite[Thm.~$3.11$]{TannoYasuda2021WildCyclic}.
Furthermore, in \cite[$\S4$]{Yamamoto2021}, the author calculates the $\mathbf{v}$-functions associated to $3$-dimensional representations of groups of the form $\ZlZ \rtimes \ZZ/3\ZZ$, $(\ZlZ)^2 \rtimes \ZZ/3\ZZ$ and $(\ZlZ)^2 \rtimes S_3$ over fields of characteristic $3$ (where $\ell$ is a prime different to $3$). 
In \cite[$\S5$]{Yamamoto2021}, the author computes the $\mathbf{v}$-function for the $2$-dimensional representation of the group $(\ZpZ)^2$ in characteristic $p$ and shows that it is not determined by the ramification filtration. 
Moreover, the computation of the $\mathbf{v}$-function for groups of the form $G = H \times \ZpZ$ for an abelian subgroup $H$ of order coprime to $p$ has been done in \cite{Fan2024EulerCrepant}. In this work, the author also computes the $\mathbf{v}$-function for the permutation representation of the alternating group $A_4$ and $(\ZZ/2\ZZ)^2$ in characteristic $2$.
\subsection{The wild Mckay correspondence}
The wild Mckay correspondence relates the $\mathbf{v}$-function of a modular $G$-representation $V$ to an invariant of the quotient variety $X = V/G$ called the \emph{stringy motive}.
In \cite[Cor.~$16.3$]{Yasuda2024WildDM}, the second named author proves that in the wild case, the stringy motive $M_{\str}(X)$ is given by the following motivic integral
\[M_{\str}(X) = \int_{\Delta_G} \mathbb{L}^{d-\mathbf{v}_V},\] 
where $\LL$ is the class of $\AAA_k^1$ in the ring of modified Grothendieck varieties (see Definition \ref{def; grothendieck ring}). 
One of the difficulties when calculating the stringy motive in the wild case is that the space $\Delta_G$ in the motivic integral is infinite dimensional.
In order to calculate $M_{\str}(X)$, one must stratify this space into countably many finite dimensional spaces. 
For $\gamma \in \ZmZ$, we let $m_{\gamma} = \frac{m}{\gcd(m,\gamma)}$ and $\gamma^{\dagger} = \frac{\gamma}{\gcd(m,\gamma)}$. We define $\gamma^{-1} \in \ZmZ$ as follows: Let $\theta: \ZZ/m_{\gamma}\ZZ \hookrightarrow \ZmZ$ be the canonical injection. Then $\gamma^{-1} = \theta(\gamma^{\dagger -1})$.
We write 
\[\ZZ_{\gamma} = \{e \in \ZZ : p \nmid e>0, e \: \text{satisfies the congruence} \:e \equiv k \gamma^{-1} \pmod{m}\}\] 
for some $k \in \{0,1,\cdots,m-1\}$ depending on $G$ and the action of $\ZmZ$ on $\ZpZ$.
For every $\gamma$ and each $r \in \ZZ_{\gamma}$, $\Delta_{G,\gamma}(r)$ is a scheme whose $k$-points are in bijection with isomorphism classes of $G$-extensions with fixed ramification jump $r$.
The spaces $\Delta_{G,\gamma}(r)$ are of finite dimension 
\[\dim \Delta_{G,\gamma}(r) = \left\lfloor \frac{r - 1}{m_{\gamma}} \right\rfloor + 1 - \left\lfloor \frac{\left\lfloor \frac{r - 1}{m_{\gamma}} \right\rfloor + 1}{p}\right\rfloor.\]
If $V$ is a $d$-dimensional representation of $G$, 
with a decomposition into $l$ indecomposable $d_{\lambda}$-dimensional representations
\[V = \bigoplus_{\lambda = 1}^{l}V_{d_{\lambda}, s_{\lambda}},\]
 we define the invariant $D_{V}$ by \[D_{V} =\sum_{\lambda = 1}^l \frac{d_{\lambda}(d_{\lambda}-1)}{2}.\]
\begin{thm}[Theorem \ref{stringy motive thm}]
    Let $X = V/G$ be a quotient variety where $V$ is a $d$-dimensional $G$-representation and for every $\gamma \in \ZmZ$, 
    let $\mathbf{v}_{V,\gamma}$ be the $\mathbf{v}$-function associated to $V$. 
    Then if $D_{V} \geq p$, we have 
    \[ M_{\str}(X)  = \sum_{g \in \ZmZ}\LL^{d - \mathbf{v}_{V}(g)}  
         + (\LL - 1) \LL^{d - 1}
         \left(\frac{\sum_{\gamma^{-1} = 0}^{m-1}\sum_{\substack{s = 0\\s \in \ZZ_{\gamma}}}^{m_{\gamma}p - 1}
         \LL^{\dim \Delta_{G,\gamma}(s)-\mathbf{v}_{V,\gamma}(s)}}
         {1-\LL^{p - 1 - D_{V}}} \right).\]
\end{thm} \noindent From this theorem we also obtain information on the stringy Euler number of the quotient variety $X$, 
a generalisation of the Euler characteristic in the smooth setting. 
\begin{prop}[Proposition \ref{prop;str euler number}]
    Let $X = V/G$ be the quotient variety, where $G = \ZpZ \rtimes \ZmZ$ and $V$ a $d$-dimensional $G$-representation. Then we have \[ e_{\str}(X) = \frac{m D_V}{D_V - p + 1}.\]
\end{prop} \noindent An interesting facet of our formula for the stringy Euler number is that it depends not just on the group $G$, but also on the choice of representation.
We highlight this phenomenon in Section \ref{subsec; S_3 examples stringy motive}, where we use the formula in Theorem \ref{stringy motive thm}
to calculate the stringy Euler number for two different representations of $S_3$
and show they do not coincide.
We also show that in the $6$-dimensional example our formula matches the motivic version of Bhargava's mass formula.
Providing the existence of a crepant resolution $f:Y \rightarrow X$, we obtain a formula for the $\ell$-adic Euler characteristic of $Y$ in the form of Batyrev's formula.
\begin{cor}
    Let $X = V/G$ and suppose there exists a crepant resolution $f: Y \rightarrow X$. Let $e_{\topo}(Y)$ denote the $\ell$-adic Euler characteristic of $Y$. Then we have
    \[e_{\topo}(Y) = \frac{m D_V}{D_V - p + 1}.\]
\end{cor}
\noindent In particular, the crepant resolution of the quotient variety in the wild case does not in general have Euler characteristic $\# \conj(G)$, and hence does not match the formula in the tame case.
This contrasts with \cite[Thm.~$1.2$]{Fan2024EulerCrepant}, which shows that when $G = H \rtimes \ZpZ$ where $H$ is an abelian normal subgroup of index $p$ and such that $p \nmid \# H$, Batyrev's formula for the Euler characteristic of a crepant resolution holds even in the wild case.
\subsection{$a$- and $b$-invariants}

The $\mathbf{v}$-function is an example of a raising function in the sense of \cite[Def.~$4.1$]{darda2025batyrevmaninconjecturedmstacks} and so has associated $a$-invariant, 
defined by  
\[a(\mathbf{v}) = \sup_{g \in \mathbf{v}(\Delta_G) \backslash \{0\}} \frac{1 + \dim(\mathbf{v}^{-1}(g))}{g}.\] 
The $a$-invariant contains information on the type of singularities of the quotient variety $X = \AAA_k^{d}/G$,
where $d$ is the dimension of the representation. 
We have that $X$ has canonical (resp. terminal) singularities if 
\[\dim \left( \int_{\Delta_G \backslash \{0\}} \LL^{-\mathbf{v}} \right) \leq -1  \: (\text{resp}. <-1),\] 
where $\int_{\Delta_G \backslash \{0\}} \LL^{-\mathbf{v}}$ is the integral defining $M_{\str}(X)$. 
Moreover, we have the relation
\[a(\mathbf{v}) \leq 1 \: (\text{resp}. <1) \iff \dim \left( \int_{\Delta_G \backslash \{0\}} \LL^{-\mathbf{v}} \right) \leq -1  \: (\text{resp}. <-1). \]
In the case of $G = \ZpZ \rtimes \ZmZ$, the $a$-invariant is given by the formula
\[a(\mathbf{v}) = \max \left(\max_{[g] \in \conj^{\tame}(G)\backslash\{ [1]\} }(\age(g))^{-1},
\max_{r \in \cup_{\gamma^{-1}=0}^{m-1}\ZZ_{\gamma}}\frac{1 + \dim \Delta_{G,\gamma}(r)}{\mathbf{v}_{V,\gamma}(r)} \right).\]
From this, we obtain the following result on the singularities of the quotient variety $X = \AAA_k^{d}/G$. 
\begin{thm}
    Let $G = \ZpZ\rtimes\ZmZ$ and $X = \AAA_k^{d}/G$.
    $X$ has canonical (resp. terminal) singularities if and only if $\age(g) \geq 1$ (resp. $> 1$)
    for all
    $[g] \in \conj^{\tame}(G)\backslash  \{[1]\}$
    and 
    $\dim \Delta_{G,\gamma}(r) - \mathbf{v}_{V, \gamma }(r) \leq -1$ (resp. $< -1$)
    for all $r \in \ZZ_{\gamma}$ in the range $\{1,\cdots m_{\gamma}p -1\}$.
\end{thm}
\noindent This shows for example that it is possible for $X$ to not have canonical singularities,
even if $\age(g) \geq 1$ for all $[g] \in \conj^{\tame}(G)\backslash \{[1] \}$. 
This is in contrast to the tame case where the two conditions are equivalent, highlighting one of the ways that wild quotient singularities behave pathologically.  

\subsection{Counting wild extensions}

Another important application of $\mathbf{v}$-functions is in counting problems studied in arithmetic statistics. 
Recently there has been much interest in the interpreting problems about number field counting via counting rational points on stacks, for example in \cite{ellenbergmaninheights, DardaYasuda2024BatyrevManinDMStacks, loughransantens}.
 However comparatively little work has been done on the case of fields of positive characteristic, especially in the wild case, owing to the difficult nature of the problem. 
 Even so, in recent years there has been increasing amounts of interest regarding the problem of counting wildly ramified field extensions, such as in the works \cite{Lagemann2015, Gundlach2026, wildcount, darda2025countingtorsorswildabelian, Potthast2026}.
An issue that arises is that many of the height functions used when counting over number fields are not height functions in the wild case.
In \cite{darda2025batyrevmaninconjecturedmstacks}, the authors set up a framework for counting field extensions of positive characteristic using stacks.
This includes a new type of height function, that is flexible enough to be used to count on wild stacks. 
In their framework, the $\mathbf{v}$-function corresponds to a local height function. 
In \cite{darda2025countingtorsorswildabelian}, the authors show the $\mathbf{v}$-function is strongly suitable, and thus is a reasonable candidate as a choice height function for counting wild $G$-torsors.
It also works as a height function in the sense of \cite[Def.~$2.11$]{ellenbergmaninheights}, using their definition of heights with respect to a given vector bundle on a stack.
Height functions coming from group representations, such as $\mathbf{v}$-functions, have previously been studied in a variety different settings.
Local heights associated to the representations of a group were considered in the works \cite{DenefLoeser2002, Yasuda2014pCyclicMcKay, WoodYasuda2015MassI, towardmotivicint}.
In the global case, they feature in the works of Dummit in \cite{dummit2014} and of the second named author in \cite{Yasuda2015}, as well as in \cite{ellenbergmaninheights}.
Although there are no global counting results in this work, we prove a result on counting local $G$-extensions with fixed ramification jump $r$, using the structure of each stratum of the moduli space $\Delta_G$.
While the problem of counting tamely ramified local extensions is rather trivial, it becomes quite difficult in the wild case.
 It can often be the case that the local counting problem itself is a major obstacle to the counting of wildly ramified global fields.
\begin{prop}[Proposition \ref{g ext prop}]
Let $\gamma \in (\ZmZ)^{\times}$ be as in Lemma \ref{Pries lem} and such that $\gamma$ satisfies $\gamma = \gamma^q$ in $\ZmZ$. 
     Then the number of $G$-extensions of $\FF_q((t))$ with given ramification jump $r \in \ZZ_{\gamma}$ is \[|Z(G)|(q-1)q^{ \left\lfloor \frac{r - 1}{m} \right\rfloor - \left\lfloor \frac{\left\lfloor \frac{r - 1}{m} \right\rfloor + 1}{p}\right\rfloor},\] where $Z(G)$ is the centre of $G$.
     Moreover, if $\gamma \neq \gamma^q$ in $\ZmZ$, there are no such $G$-extensions.
\end{prop}\noindent  Related problems of counting wildly ramified local extensions with prescribed ramification have been treated in \cite{Krasner1966, PauliSinclair2017, Monnet2025, Muller2023}.
The above formula may be helpful in the case of global counting.
\subsection{Structure}In Section \ref{sec:modular reps} we collect results on the modular representation theory of groups of the form $G = \ZpZ \rtimes \ZmZ$ over fields of characteristic $p$ needed for our computation of the $\mathbf{v}$-function. 
In Section \ref{sec; moduli space structure} we explicitly describe the structure of the moduli space of $G$-torsors, and prove the result on the number of $G$-extensions with fixed ramification jump.
In Section \ref{sec; v functions} we compute the $\mathbf{v}$-function for the connected and non-connected torsors, proving the full version of Theorem \ref{defn of v function}.
In Section \ref{sec; Stringy motives} we prove the formula for the stringy motive in Theorem \ref{stringy motive thm}, 
and the formula for the stringy Euler number.
We then explicitly calculate the stringy motive for two different representations of $S_3$ over a field of characteristic $3$, illustrating the dependence of the stringy Euler number on the representation. 
Finally, in Section \ref{sec; a and b invs} we prove the formula for the $a$-invariant of our function,
and Theorem \ref{thm; quotient singularities}, 
as well as finding the $a$- and $b$-invariants for both representations of $S_3$.
\subsection{Notation}\label{sec; notation}

In this paper we let $k$ be an algebraically closed field of characteristic $p>0$, and write $K = k((t))$ for the field of formal Laurent power series over $k$. We use the terminology $G$-extension of $K$ to denote a finite extension $L/K$ whose Galois group is isomorphic to $G$.  For a field $F$ of characteristic $p$, we define the Artin-Schreier map \[\wp: F \rightarrow F, \quad x \mapsto x^p - x,\] which is additive under the assumption that $\text{char} F = p$. Furthermore, we fix a choice of $m$-th root of unity $\zeta_m \in k$ that satisfies $\zeta_{m'n}^n = \zeta_{m'}$. We will abbrieviate ``Deligne-Mumford stack" to ``DM stack".

\subsection{Acknowledgements}
 Tavernier's work was partially supported by the JSPS Postdoctoral Fellowship for Research in Japan and partially supported by the EPSRC Doctoral Studentship. Yasuda was supported by JSPS KAKENHI Grant Numbers JP21H04994, JP23K25767, and JP24K00519. The authors would like to thank Linghu Fan and Sophi Marques for helpful discussions.

\section{Representations of wild-by-tame split metacyclic groups}\label{sec:modular reps}

\noindent In this section we collect some results from modular representation theory to describe the indecomposable representations of $G = \ZpZ \rtimes \ZmZ$ over an algebraically closed field of characteristic $p$.
We say that a module $U$ over a ring $R$ is indecomposable if $U \cong V \oplus W$ implies that $V = 0$ or $W = 0$. 
As we are working over a field whose characteristic divides $\# G$,
there exist non-simple indecomposable modules. 
However, to find all the indecomposable $kG$-modules we begin by describing the simple modules.

\begin{lem}\label{lem:simple mods}
    The simple $kG$-modules are exactly the simple $k\ZmZ$-modules.
    In particular there are $m$ simple $kG$-modules.
\end{lem}

\begin{proof}
By \cite[Cor.~$6.2.2$]{Webb_2016}, the common kernel of all the simple $G$-modules is the (unique) largest normal $p$-subgroup of $G$, which in our case is $\ZpZ$. In particular, the simple $kG$-modules are precisely the simple $k\ZmZ$ modules, made into $kG$-modules via the quotient homomorphism $G \rightarrow G/(\ZpZ)$. Futhermore, $m$ is coprime to the characteristic of $k$, and so $k\ZmZ$ is semisimple and there are exactly $m$ simples.
\end{proof}

\noindent By \cite[Page.~$31$,~Thm.~$3$]{Alperin_1986} there is a one-to-one correspondence between the isomorphism classes of indecomposable projective $kG$-modules and the isomorphism classes of simple $kG$-modules. This correspondence comes from associating to each indecomposable projective $P$ the quotient $P/\rad P$, where $\rad P$ is the radical of $P$, which is the intersection over all maximal submodules of $P$. 

The remaining non-simple, non-projective indecomposables may be found by considering the structure of the Brauer tree. By \cite[Exa.~$5.1.3$]{craven}, the Brauer tree of $kG$ is a star, and so by \cite[Exa.~$5.2.6$]{craven} every indecomposable $kG$-module is a quotient of an indecomposable projective module. In particular, every indecomposable $kG$-module is uniserial and has structure entirely determined by the structure of the projective indecomposables. 

Finally, we calculate the number of indecomposable modules of a group of the form $\ZpZ \rtimes\ZmZ$.
\begin{lem}\label{lem:no_indecomp_kG_modules}
    Let $k$ be a field of prime characteristic $p$ containing all the $m$-th roots of unity and $G = \ZpZ \rtimes\ZmZ $ for some $m \in \NN$ coprime to $p$. Then there are exactly $m \cdot p$ indecomposable $kG$-modules.
\end{lem}

\begin{proof}
Denote by $l_k(\ZmZ)$ the number of isomorphism classes of simple $k\ZmZ$ modules. By \cite[Cor.~$11.2.2$]{Webb_2016}, since $G = \ZpZ\rtimes \ZmZ$ and $m$ is coprime to $p$, the number of isomorphism classes of indecomposable $kG$-modules is given by $p\cdot l_k(\ZmZ)$. As $k\ZmZ$ is semisimple, we have $l_k(\ZmZ) = m$.
\end{proof}

\subsubsection{Matrix representations}\label{matrix rep subsec} We will need an explicit matrix representation for the actions of the generators of $\ZpZ$ and $\ZmZ$.
We begin by setting up some notation to be used throught the paper.
Let $\sigma$ be such that $\ZpZ = \langle \sigma \rangle$ and $\tau$ be such that $\ZmZ = \langle \tau \rangle$. 
These must satisfy the relations $\sigma^p = \id$, $\tau^m = \id$ and $\tau \sigma \tau^{-1} =\sigma^{a(G)}$ for some $a(G) \in \left( \ZpZ\right)^{\times}$ which determines the representation. Note that $a(G)$ satisfies $a(G)^m \equiv 1 \bmod p$.
Consider the kernel of the group action, which is the subgroup $\{ g \in \langle \tau \rangle : g \sigma g^{-1} = \sigma\}$ of $\ZmZ$. 
This is exactly the subgroup of elements which act trivially on $\ZpZ$.
Denote by $n^{\dagger}$ the cardinality of this subgroup.
The $a(G)$ appearing in the description of the group action may be identified with an element $a(G) \in \FF_p^{\times}$, and its order is $n = m/n^{\dagger}$.
We will also need the following definition. \begin{defi} \label{def; h invariant}
    Let $\zeta_{p - 1}$ be the chosen root of unity from Section \ref{sec; notation}, and write $p-1 = n n_1$. We set $h'(G) \in \{1, \cdots , p-1 \}$ to be the element such that $a(G) = \zeta_{p-1}^{h'(G)}$. The \emph{invariant} of $G$ is $h(G) = h'(G)/n$.
\end{defi}
\noindent For a $d$-dimensional indecomposable $G$-representation, we choose a suitable basis $x_1,\cdots,x_d$ such that the generator $\sigma$ of $\ZpZ$ is represented by a matrix that is a single Jordan block of size $d$, which is to say a matrix of the form 
\[\begin{pmatrix} 1 & 1 & \cdots & 0 \\
\vdots & 1 & 1 \cdots & 0 \\ \vdots & \vdots& \ddots &  \vdots \\ 0 & 0 & \cdots & 1 \end{pmatrix}_{d \times d}.\] 
It then follows from the fact that $\sigma$ maybe written in this way that the matrix for $\tau$ is upper triangular. 
\begin{lem}\label{lem;tau is upper triangular}
    Let $\sigma$, $\tau$ be as above.
    Then the matrix representing $\tau$ is of the form
    \[\begin{pmatrix} \xi_1 & * & \cdots & * \\
\vdots & \xi_2 & * \cdots & \vdots \\ \vdots & \vdots& \ddots &  * \\ 0 & 0 & \cdots & \xi_d \end{pmatrix}_{d \times d},\] where the $\xi_i$ are $m$-th roots of unity. 
In particular, it is upper triangular.
\end{lem}
\begin{proof}
    Set $\delta$ to be the operator $\delta := \sigma- \id$.
    It is enough to show that the subspaces $\{\delta^i = 0\}$ are preserved by $\tau$ for every $i$.
    Write $\delta^{a(G)} = (\sigma- \id)^{a(G)}$.
    By assumption, $\tau$ satisfies the equality $ \delta = \tau^{-1} \delta^{a(G)} \tau$.
    Furthermore, for each $i$ we have $\ker(\delta^i) = \ker((\delta^{a(G)})^i)$.
    Thus, $\ker(\delta^i)$ is preserved by $\tau$ for every $i$.
    We set the first diagonal entry to be an $m$-th root of unity $\xi_1$.
    Then every other entry may be determined by $\xi_1$ via the relation $ \delta = \tau^{-1} \delta^{a(G)} \tau$ and in particular,
    they are all $m$-th roots of unity.
\end{proof}
\section{The moduli space of $G$-torsors}\label{sec; moduli space structure}

\subsection{Explicit description of the $G$-actions}

  Let $\sigma$ be such that $\ZpZ =\langle \sigma \rangle$ and $\tau$ be such that $\ZmZ = \langle \tau \rangle$ and $V_{d,s}$ be a $d$-dimensional indecomposable representation of $G$ with coordinate ring $k[x_1,x_2,\cdots ,x_d]$. We may choose these coordinates such that the action of $\sigma$ is given by \[\sigma (x_i) = \begin{cases}
        x_i + x_{i+1} & \text{if} \: i<d, \\
        x_d & \text{if} \: i=d.
    \end{cases}\] In other words, we take the Jordan normal form of $\sigma$ as described in $\S$\ref{matrix rep subsec}.
    
    On the other hand, let $L/K$ be a $G$-extension. There is an intermediate extension $L/F/K$ such that $F = L^{\sigma}$, $F/K$ is a $\ZmZ$-extension generated by $\alpha = t^{\frac{1}{m}}$ and $L/F$ is a $\ZpZ$-extension generated by $\beta$, where $\beta$ satisfies the Artin-Schreier relation $\beta^{p} - \beta +f = 0$. The $G$-extension $L/K$ is generated by $\alpha^{i}\beta^{j}$ for $i = 0,\cdots , m-1$ and $j = 0,\cdots, p-1$. By \cite[Lem.~ $1.4.1(ii)-(iii)$]{Pries2002Families}, we may assume that \[\sigma(\alpha) = \alpha, \quad \sigma(\beta) = \beta + 1, \quad \tau(\alpha) = \zeta_m^{\gamma} \alpha, \quad \tau(\beta) = \zeta_m^{-\gamma r} \beta, \] where $\gamma \in ( \ZZ/m\ZZ)^{\times}$ and $r$ is the ramification jump. We can use the ramification jump and $\gamma$ appearing in the action of $\tau$ to describe the eigenvalues of $\tau$.
    \begin{lem}\label{lem; matrix of tau}
        Let \[ \tau = \begin{pmatrix} \xi_1 & * & \cdots & * \\
\vdots & \xi_2 & * \cdots & \vdots \\ \vdots & \vdots& \ddots &  * \\
0 & 0 & \cdots & \xi_d \end{pmatrix}_{d \times d},\] where $\xi_1 = \zeta_m^s$. Then the roots of unity appearing on the middle diagonal satisfy the relation $\xi_i = \xi_{i - 1}a(G)^{-1}$. In particular the eigenvalues of $\tau$ are \[\zeta_m^s, \zeta_m^{s - \gamma r}, \zeta_m^{s - 2\gamma r}, \cdots, \zeta_m^{s - (d-1)\gamma r}.\] 
    \end{lem}

    \begin{proof}
       We know that $\tau$ and $\sigma$ satisfy $\sigma = \tau^{-1}\sigma^{a(G)}\tau$ where 
\[ \sigma^{a(G)} =  \begin{pmatrix} 1 & a(G) & \cdots & * \\
\vdots & 1 & a(G) \cdots & \vdots \\ \vdots & \vdots& \ddots &  a(G) \\
0 & 0 & \cdots & 1 \end{pmatrix}_{d \times d}.\] It follows from a matrix calculation that the diagonal entries of $\tau$ satisfy $\xi_i = \xi_{i-1}a(G)^{-1}$. In particular they are of the form \[\xi_2 = \xi_1a(G)^{-1}, \xi_3 = \xi_1a(G)^{-2}, \cdots, \xi_d a(G)^{-(d-1)}.\]By \cite[Lem,~$1.4.1(iii)$]{Pries2002Families} we have $a(G) = \zeta_{m}^{\gamma r}$ for the chosen root of unity $\zeta_m$, and since $\xi_1$ is given by $\xi_1 = \zeta_m^s$, we have that the eigenvalues of $\tau$ are \[\zeta_m^s, \zeta_m^{s - \gamma r}, \zeta_m^{s - 2\gamma r}, \cdots, \zeta_m^{s - (d-1)\gamma r}.\]
    \end{proof} \noindent If $\gamma \in \ZmZ$ is such that $\gcd(\gamma, m) = g >1$, we may write $\gamma = \gamma^{\dagger} \cdot g$ and $m = m_{\gamma}\cdot g$. By the choice of root of unity, we have the action of $\tau$ on $ \alpha$ is given by \[\tau(\alpha) = \zeta_m^{\gamma}\alpha = \zeta_{m_{\gamma}\cdot g}^{\gamma^{\dagger} \cdot g}\alpha = \zeta_{m_{\gamma}}^{\gamma^{\dagger}}\alpha,\] and $\gamma^{\dagger}$ and $m_{\gamma}$ are coprime. \begin{defi}
   A \emph{$G$-\'etale $K$-algebra} is a finite \'etale $K$-algebra $M$ of degree $\# G$ that is endowed with a $G$-action which satisfies that $M^{G} = \mathcal{O}_K$.
   Denote by $G\et(K)$ the set of isomorphism classes of $G$-\'etale $K$-algebras. 
   We can write any $M\in G\et(K)$ as $M = F^{\oplus n}$ for a $H$-extension $F/K$, where $H$ is a subgroup of $G$.  
\end{defi}
    \begin{nota}
          For a $k$-algebra $M$ with a $G$-action, let $\delta = \sigma - \id_M$, and for any integer $n$ write $M^{\delta^n=0}$ to mean the kernel of the $k$-linear operator $\delta^{n}:M \rightarrow M$.
          We will set $\delta^0 := \id$.
    \end{nota}
  \noindent We can describe the $G$-action on the coordinates using $\delta$ as follows: \[\delta(x_i) = \begin{cases}
        x_{i+1} & \text{if} \: i<d, \\
        0 & \text{if} \: i=d.
    \end{cases}\] 
\noindent It can be helpful to work with $\delta$ rather than $\sigma$ when studying the $G$-action.
\begin{prop}\label{prop:basis for L}
    For any $1 \leq j \leq d-1$ and $0 \neq h \in k((t))$ we have that $\delta^j(\alpha^i \beta^j h) \neq 0$ and $\delta^{j+1}(\alpha^i \beta^j h) = 0$. 
    Hence we obtain the following basis for $L^{\delta^{d}= 0}$: \[L^{\delta^{d}= 0} = \bigoplus_{i = 0 }^{m-1} \bigoplus_{j = 0}^{d-1}  k((t)) \cdot \alpha^i \beta^j. \]
\end{prop}

\begin{proof}
    We proceed by induction on $j$. When $j = 1$ we clearly have  \[\delta^{1}(\alpha^i h \beta^{1}) = \alpha^i h(\beta + 1) - \alpha^i h\beta = \alpha^i h \] and \[\delta^{2}(\alpha^i h \beta^{1}) = \delta(\alpha^i h) = 0. \] Assume this holds for some $j = n$. Then \begin{align*}
         \delta^{n+1}(\alpha^i h \beta^{n+1}) = & \delta^{n}(\sigma(\alpha^i h \beta^{n+1}) - \id(\alpha^i h \beta^{n+1})\\
        = & \delta^{n} \left( \alpha^i h \sum_{\ell =0}^{n+1} \binom{n+1}{\ell} \beta^{n+1-\ell} - \alpha^i h \beta^{n+1} \right) \\
        = & \delta^{n} \left(  \alpha^i h \sum_{\ell =1}^{n+1} \binom{n+1}{\ell} \beta^{n+1- \ell} \right) \neq 0,
    \end{align*} and \[\delta^{n+2}(\alpha^i h \beta^{n+1}) =0. \]
\end{proof}

\begin{cor}\label{basis for OL}
    Let $\mathcal{O}_L$ be the integer ring of $L$ and $v_L$ be the normalised valuation on $\mathcal{O}_L$. Then we have \[\mathcal{O}_L^{\delta^{d}= 0} = \bigoplus_{i = 0}^{m-1} \bigoplus_{j = 0}^{d-1}k[[t]]\cdot \alpha^i \beta^j t^{n_{i,j}} \] where \[ n_{i,j} = \left\lceil \frac{-iv_{L}(\alpha) - jv_{L}(\beta)}{e_{L/K}} \right\rceil,\] and $e_{L/K}$ is the ramification index of $L/K$.
\end{cor}

\begin{proof}
    By definition, we have $v_L(\alpha^i\beta^j)\not\equiv v_L(\alpha^{i'}\beta^{j'}) \bmod p$ for all $(i,j) \neq (i',j')$ where $(i,j)$ runs through $\{0,1,\cdots, m-1\} \times \{0,1,\cdots, d-1\}$. Then the result follows from intersecting the basis from Proposition \ref{prop:basis for L} with $\mathcal{O}_L$.
\end{proof}

\subsection{Moduli space of $G$-torsors}\label{moduli space sec}

To calculate the stringy motive, we must integrate over the space $\Delta_G$.
T is the moduli space of $G$-torsors over the formal punctured disk $\spec k((t))$.
It is the space that the $\mathbf{v}$-function is defined on, and it is infinite-dimensional.
In order to compute the motivic integral, we must gain an understanding of the structure of this space. 
This may be done by decomposing it into a disjoint union of countably many subspaces of finite dimension.
We first decompose it into spaces of wild and tame torsors, and further into the connected and non-connected torsors. 
Let \[\Delta_G =  \Delta_G^{\tame} \sqcup \Delta_G^{\wild},\] where $\Delta_G^{\tame}$ is the moduli space of $G$-torsors whose wild part is unramified, and $ \Delta_G^{\wild}$ is the moduli space of $G$-torsors such that the wild part is ramified.
The space $\Delta_G^{\tame}$ is zero-dimensional as these torsors are tamely ramified.
They correspond to the sets of $G$-\'etale $K$-algebras \begin{align*}D^{\tame} = & \{L \cong Q^{p}:Q \:\text{an \'etale $\ZmZ$-algebra}\} \\
  = &  \bigcup_{q \mid m}\{L \cong Q^{\oplus\frac{mp}{q}}: Q/K \: \text{a $\ZZ/q\ZZ$-extension}\}\end{align*}
  where the union is over  divisors $q$ of $m$. 
  $\Delta_G^{\tame}$ also contains the fully unramified torsors, corresponding to the set of unramified $G$-algebras
  \[D^{\unr} = \{L \cong k((t))^{mp} \}.\] 
  However, the space $\Delta_G^{\wild}$ is infinite-dimensional. 
  We stratify it using the \emph{ramification jump}. 
  \begin{defi}
    Let $v_L$ be the natural valuation on $L$ and $\pi$ a uniformiser.
    The $i$-th ramification group (in lower numbering) is the normal subgroup of all $g \in G$ such that $v(\pi(g) - \pi) \geq i + 1$. 
    For $ 1 \neq g \in G$, the (lower) ramification jump is the non-negative integer $r$ such that $v(g(\pi) - \pi) = r + 1$.
\end{defi}
\noindent The space $\Delta_G^{\wild}$ contains both the space of connected torsors, 
and some non-connected torsors.
The connected torsors correspond to $G$-extensions of $K$, which is
the set of algebras
\[D^{\text{field}} = \{ L/K \: \text{a $\ZpZ \rtimes \ZmZ$-extension}\}.\]
  The following result from \cite[Lem.~$1.4.1$]{Pries2002Families} describes the properties of the ramification jump of $G$-extensions of this form. 
  \begin{lem}\label{Pries lem}
    Let $G = \ZpZ \rtimes \ZmZ$, and $\sigma$ and $\tau$ be the generators of $\ZpZ$ and $\ZmZ$ respectively, 
    and $a(G) \in (\ZpZ)^{\times}$ be such that $\tau \sigma \tau^{-1} =\sigma^{a(G)}$.
    If $r$ denotes the ramification jump then the following hold: \begin{enumerate}
        \item The ramification jump satisfies $\gcd(r,p) = 1$.
        \item The possible degrees of the terms appearing in the Laurent series $f = \beta^p - \beta$ are all congruent to $r$ modulo $m$.
        \item We have that
        $n^{\dagger} = \# \{ g \in \langle \tau \rangle : g \sigma g^{-1} = \sigma\}$
        satisfies $n^{\dagger} = \gcd(r,m)$ and furthermore, we have \begin{equation}\label{congruence condition}
            r \equiv \gamma^{-1} h(G) n^{\dagger} \pmod m,
        \end{equation} where 
        $\gamma \in (\ZmZ)^{\times}$ 
        is given by the action of $\tau$ on the generators $\alpha$ and $\beta$ of the $G$-extension and $h(G)$ is the invariant from Definition \ref{def; h invariant}.
    \end{enumerate}
\end{lem}
\noindent For every $\gamma \in (\ZmZ)^{\times}$, we write $\ZZ_{\gamma}$ for the set 
\[\ZZ_{\gamma} = \{e \in \ZZ : p \nmid e>0, e \: \text{satisfies the congruence} \: e \equiv \gamma^{-1} h(G) n^{\dagger} \pmod m\}\] 
where $h(G)$ and $n^{\dagger}$ are as in Lemma \ref{Pries lem}.
Note that the congruence condition defining $\ZZ_{\gamma}$ is equivalent to asking that $e \equiv r \pmod m$ by Lemma \ref{Pries lem}$(3)$.
Then we let $E(G,r,\gamma)$ be the set
\[E(G,r,\gamma) = \{e \in \ZZ_{\gamma} : 1 \leq e \leq r\}.\] 
The congruence condition $r \equiv \gamma^{-1} h(G) n^{\dagger} \pmod m$ is a condition on both $r$ and $\gamma \in (\ZmZ)^{\times}$.
\begin{prop}\label{prop: dim of moduli space}
   Let $\gamma \in (\ZmZ)^{\times}$. There exists a scheme $\Delta_{G,\gamma}(r)$ such that the $k$-points of $\Delta_{G,\gamma}(r)$ are in bijection with the isomorphism classes of $\ZpZ \rtimes \ZmZ$-extensions with ramification jump $r$. Furthermore, for each $r \in \ZZ_{\gamma}$, the dimension of $\Delta_{G,\gamma}(r)$ is given by
    \begin{equation}\label{eq; dim of moduli space}
        \# E(G,r,\gamma) = \left\lfloor \frac{r - 1}{m} \right\rfloor + 1 - \left\lfloor \frac{\left\lfloor \frac{r - 1}{m} \right\rfloor + 1}{p}\right\rfloor.
    \end{equation}
\end{prop}

\begin{proof}
    The first part follows from \cite[Lem.~$5.5$]{OBUS2010565} and \cite[Thm.~$5.6$]{OBUS2010565}. To calculate the dimension, we proceed as follows. We count the integers $1 \leq e \leq r$ of the form $e = r - mt$ for $t \in \ZZ_{\geq 0}$. As $e \geq 1$, $t$ must satisfy $t \leq \frac{r - 1}{m}$. In particular, $t$ takes values in $0,1, \cdots, \left\lfloor \frac{r - 1}{m} \right\rfloor$, and thus there are $\left\lfloor \frac{r - 1}{m} \right\rfloor + 1$ possible values for $e$ satisfying the congruence condition. Then we must exclude the $e$ that are divisible by $p$. These correspond to integers $t \equiv m^{-1}r \pmod m$. There are \[\left\lfloor \frac{\left\lfloor \frac{r - 1}{m} \right\rfloor + 1}{p}\right\rfloor \] such values of $t$. Therefore it follows that \[\# E(G,r,\gamma ) = \left\lfloor \frac{r - 1}{m} \right\rfloor + 1 - \left\lfloor \frac{\left\lfloor \frac{r - 1}{m} \right\rfloor + 1}{p}\right\rfloor,\] for each $\gamma \in (\ZmZ)^{\times}$
\end{proof}\noindent Then for every $\gamma \in (\ZmZ)^{\times}$ and ramification jump $r$, $\Delta_{G,\gamma}$ admits the stratification 
\[\Delta_{G,\gamma} =\bigsqcup_{r \in  \ZZ_{\gamma}}\Delta_{G,\gamma}(r).\] 
\noindent The space $\Delta_G^{\wild}$ also contains some non-connected torsors.
These non-connected torsors correspond to the case where the $\gamma$ appearing in the action of $\tau$ is not invertible mod $m$.
In this case, Lemma \ref{Pries lem} does not hold as stated.
However, it is possible to reduce to the connected case.
Let $m_{\gamma} = \frac{m}{\gcd (\gamma, m)}$.
The map $\Gamma_k \rightarrow G$ factors through \[\Gamma_k \twoheadrightarrow \ZpZ \rtimes \ZZ/m_{\gamma}\ZZ \hookrightarrow \ZpZ \rtimes \ZmZ. \] 
For every $\gamma$, we will denote by $\Delta_{G, \gamma}$ the moduli space of $G$-torsors corresponding to the set of algebras
\[D_{\gamma} = \{L \cong E^{\gcd(\gamma, m)}: E/K \:\text{a $\ZpZ \rtimes \ZZ/m_{\gamma}\ZZ$-extension} \}.\]
When $\gamma \in (\ZmZ)^{\times}$, this is just the set of algebras $D^{\text{field}}$.
Let $\tau_{\gamma}$ be such that $\ZZ/m_{\gamma}\ZZ = \langle \tau_{\gamma} \rangle $, and $\alpha_{\gamma}$ generate the $\ZZ/m_{\gamma}\ZZ$-part of the $\ZpZ \rtimes \ZZ/m_{\gamma}\ZZ$-extension. 
We may thus assume $\tau_{\gamma}$ acts on $\alpha_{\gamma}, \beta$ by \[\tau_{\gamma}(\alpha_{\gamma}) = \zeta_{m_{\gamma}}^{\gamma^{\dagger}}\alpha_{\gamma} \quad \tau_{\gamma}(\beta) =\zeta_{m_{\gamma}}^{-\gamma^{\dagger}r} \beta\] for $\gamma^{\dagger}  \in  (\ZZ/m_{\gamma}\ZZ)^{\times}$.
Lemma \ref{Pries lem} holds for $\ZpZ \rtimes \ZZ/m_{\gamma}\ZZ$-extensions, as $\gamma^{\dagger}$ is invertible mod $m_{\gamma}$.
Moreover, we have that \[\tau_{\gamma}(\alpha_{\gamma}^i\beta^j) = \zeta_{m_{\gamma}}^{\gamma^{\dagger}i -\gamma^{\dagger}rj}. \]
Let $\zeta_{m_{\gamma}}^{s_{\gamma}}$ be the first diagonal entry of $\tau_{\gamma}$.
We have $\zeta_{m_{\gamma}}^{\gamma^{\dagger}i -\gamma^{\dagger}rj} = \zeta_{m_{\gamma}}^{s_{\gamma}}$ if and only if \[i - rj \equiv s_{\gamma}\gamma^{\dagger -1} \pmod{m_{\gamma}}.\]
Then let $\theta$ be the injection $\theta: \ZZ/m_{\gamma}\ZZ  \hookrightarrow \ZmZ $. 
We define $\gamma^{-1} \in \ZmZ$ to be the image $\theta(\gamma^{\dagger -1})$.
In particular the set 
\[\ZZ_{\gamma} = \{e \in \ZZ : p \nmid e>0, e \: \text{satisfies the congruence} \: e \equiv \gamma^{-1} h(G) n^{\dagger} \pmod{m}\}\] 
is well defined. 
Hence there is a stratification 
\[\Delta_G^{\wild} = \bigsqcup_{\gamma^{-1}=0}^{m-1}\bigsqcup_{r \in \ZZ_{\gamma}}\Delta_{G, \gamma}(r).\] 
From this we obtain the following corollary to Proposition \ref{prop: dim of moduli space}.
\begin{cor}\label{cor;non connected dim}
   For each $\gamma \in \ZmZ$, let $m_{\gamma} = \frac{m}{\gcd(\gamma, m)}$. The spaces $\Delta_{G,\gamma}(r)$ of non-connected $G$-torsors each have dimension 
   \[\dim \Delta_{G,\gamma}(r) =  \left\lfloor \frac{r - 1}{m_{\gamma}} \right\rfloor + 1 - \left\lfloor \frac{\left\lfloor \frac{r - 1}{m_{\gamma}} \right\rfloor + 1}{p}\right\rfloor. \]
\end{cor}

\begin{proof}
    For every such $\gamma$, we can reduce to the case of connected $\ZpZ \rtimes \ZZ/m_{\gamma}\ZZ$-torsors. 
    Then if we write $H_{\gamma} = \ZpZ \rtimes \ZZ/m_{\gamma}\ZZ$, the inclusion 
    $\ZpZ \rtimes \ZZ/m_{\gamma}\ZZ \hookrightarrow \ZpZ \rtimes \ZZ/m\ZZ$ 
    induces the quasifinite morphism 
    $\Delta_{H_{\gamma}} \rightarrow \Delta_G$ 
    which preserves the dimension of each locus with ramification jump $r$.
    Then we have
    \[  \dim \Delta_{G,\gamma}(r) = \dim \Delta_{H_{\gamma}}(r) =  \left\lfloor \frac{r - 1}{m_{\gamma} } \right\rfloor + 1 - \left\lfloor \frac{\left\lfloor \frac{r - 1}{m_{\gamma} } \right\rfloor + 1}{p}\right\rfloor.\]
\end{proof}\noindent When $\gamma = 0$, the action of $\tau$ on $\alpha$ and $\beta$ is trivial and we are in the case of purely wild $\ZpZ$-torsors. In this case $g = m$ and the formula in Corollary \ref{cor;non connected dim} reduces to $r - \left\lfloor \frac{r}{p} \right\rfloor$, agreeing with the formula in \cite[Prop.~$2.11$]{Yasuda2014pCyclicMcKay}.
The dimension of each $\Delta_{G,\gamma}(r)$ satisfies the following relation.

\begin{lem}\label{lem: change of var dim}
    We have that $\dim \Delta_{G,\gamma}((m_{\gamma}p)\cdot n + s) = \dim\Delta_{G,\gamma}(s) + n(p-1)$.
\end{lem}

\begin{proof}
    We have \begin{align*}
        \left\lfloor \frac{(m_{\gamma}p)\cdot n + s - 1}{m_{\gamma}} \right\rfloor + 1  - & \left\lfloor \frac{\left\lfloor \frac{(m_{\gamma}p)\cdot n + s - 1}{m_{\gamma}} \right\rfloor + 1}{p}\right\rfloor \\
        = & \left\lfloor \frac{s - 1}{m_{\gamma}} \right\rfloor + np + 1-  \left\lfloor \frac{\left\lfloor \frac{s - 1}{m_{\gamma}} \right\rfloor + np + 1}{p}\right\rfloor \\
        = & \left\lfloor \frac{s - 1}{m_{\gamma}} \right\rfloor + np + 1 - n -  \left\lfloor \frac{\left\lfloor \frac{s - 1}{m_{\gamma}} \right\rfloor + 1}{p}\right\rfloor \\
        = & n(p-1) +  \left\lfloor \frac{s - 1}{m_{\gamma}} \right\rfloor + 1 - \left\lfloor \frac{\left\lfloor \frac{s - 1}{m_{\gamma}} \right\rfloor + 1}{p}\right\rfloor.
    \end{align*}
\end{proof}
\noindent By Lemma \ref{Pries lem}, the Laurent series $f$ coming from the Artin-Schreier relation $\beta^p - \beta = f$ satisfies that the possible degrees of the terms are all congruent to $r$ modulo $m$.
The representative polynomials $f \in k[t^{-1}]$ of order $-r$ are of the form 
\[f = \sum_{ i\in E(G,r,\gamma )}f_{-i}t^{-i}.\] 
Let $n^{\dagger} = \# \{ g \in \langle \tau \rangle : g \sigma g^{-1} = \sigma\}$, and $\mu_{m/n^{\dagger}}$ be the group of $m/n^{\dagger}$-th roots of unity.
Each root of unity $\zeta_{m/n^{\dagger}} \in \mu_{m/n^{\dagger}}$ acts on representative polynomials via the diagonal action $f \mapsto a(G)^{-1}f$, and on $ k^{\times} \times k^{\dim \Delta_{G,\gamma}(r) - 1}$ via multiplication on each coordinate.
\begin{lem}\label{lem; correspondence}
 There is a correspondence \[ \Delta_{G,\gamma}(r)\leftrightarrow RP_{k,r}/\mu_{m/n^{\dagger}} \leftrightarrow (k^{\times} \times k^{\dim \Delta_{G,\gamma}(r) - 1})/\mu_{m/n^{\dagger}}.\]
\end{lem}
\begin{proof}
Let $x \in \Delta_{G,\gamma}(r)$ be a $G$-torsor with fixed ramification jump $r$.
Then let $\beta$ be the generator of the $\ZpZ$-part of $x$ satisfying the Artin-Schreier relation $\beta^p - \beta -f = 0$.
We want to show that $f$ is uniquely determined up to multiplication by an element of $\mu_{m/n^{\dagger}}$.
The generator $\sigma$ of $\ZpZ$ satisfies $\tau \sigma \tau^{-1} = \sigma^{a(G)}$ and acts on $\beta$ via $\sigma(\beta) = \beta + 1$.
The invertible element $a(G) \in (\ZpZ)^{\times}$ is of order $m/n^{\dagger}$, and so we may substitute $\beta $ and $f$ in the equation by $a(G)^{-1}\beta$ and $a(G)^{-1}f$.
Since $\beta \mapsto \beta^p$ is a ring homomorphism and $\sigma^{a(G)}(a(G)^{-1}\beta) =a(G)^{-1} \beta +1$, this gives the relation $a(G)^{-1}\beta^p - a(G)^{-1}\beta - a(G)^{-1}f=0$.
As each term in $f$ has degree congruent to $r \pmod{m}$, substituting $t$ for $\zeta_m^{\gamma}t$ in $f$ scales each term by $\zeta_m^{- \gamma r}$.
By \cite[Lem.~$1.4.1$(iv)]{Pries2002Families} we have $n^{\dagger}$ satisfies that $n^{\dagger} = \gcd(r,m)$, and so $\zeta_m^{-\gamma r} \in \mu_{m/n^{\dagger}}$.
The second correspondence sends $\sum_{ i\in E(G,r,\gamma)}f_{-i}t^{-i}$ to $(f_{-i})_{i < r}$.
Since $\dim \Delta_{G,\gamma}(r)= \# E(G,r, \gamma)$, the result follows.
\end{proof}
\noindent We have that
\[ (k^{\times} \times k^{\dim \Delta_{G,\gamma}(r) - 1})/\mu_{m/n^{\dagger}}\] 
may be viewed as the variety 
\[(\mathbb{G}_m \times \AAA_{k}^{\dim \Delta_{G,\gamma}(r) - 1})/\mu_{m/n^{\dagger}}.\] 
This is not a coarse moduli space for $\Delta_{G, \gamma}(r)$, but its ind-perfection, the inductive limit taken with respect to the Frobenius morphism, is.
For our purposes, it is enough to use this as a parameter space for $\Delta_{G, \gamma}(r)$.

\subsection{Counting local extensions}
A consequence of our formula for the structure of each stratum of the moduli space is that we are able to prove a result on the number of $G$-extensions of $\FF_q((t))$ with fixed ramification jump $r$.

\begin{prop}\label{g ext prop}
     Let $\gamma \in (\ZmZ)^{\times}$ be as in Lemma \ref{Pries lem} and such that $\gamma$ satisfies $\gamma = \gamma^q$ in $\ZmZ$. 
     Then the number of $G$-extensions of $\FF_q((t))$ with given ramification jump $r \in \ZZ_{\gamma}$ is \[|Z(G)|(q-1)q^{ \left\lfloor \frac{r - 1}{m} \right\rfloor - \left\lfloor \frac{\left\lfloor \frac{r - 1}{m} \right\rfloor + 1}{p}\right\rfloor},\] where $Z(G)$ is the centre of $G$.
     Moreover, if $\gamma \neq \gamma^q$ in $\ZmZ$, there are no such $G$-extensions.
\end{prop}

\begin{proof}
Denote by $\Delta_{G,\gamma}(\leq r)$ the stack parametrising $G$-torsors with ramification jump at most $r$. 
When $\gamma \in (\ZmZ)^{\times}$, the $G$-torsors parametrised by this stack exactly correspond to $G$-extensions. We want to count $\# \Delta_{G,\gamma}(\leq r)(\FF_q) - \# \Delta_{G,\gamma}(\leq r-1)(\FF_q)$, providing that $\Delta_{G,\gamma}(\leq r)$ is defined over $\FF_q$.
 Let $W$ be the connected component of $\Delta_{\ZmZ}$ corresponding to $\gamma$. We may identify $\Delta_{G,\gamma}$ with the fibre product $\mathcal{Y} = \Delta_{G} \times_{\Delta_{\ZmZ}} W$. 
 Then $\Delta_{G,\gamma}$ is an ind-DM stack in the sense of \cite{ToniniYasuda2020I}, and hence may be written as a limit $\mathcal{Y} = \lim_n \mathcal{Y}_n$.
 We use this to show that $\Delta_{G,\gamma}(\leq r)$ is an ind-DM stack of finite type defined over $\FF_q$ by showing that it is the ind-perfection of a closed substack $\mathcal{Z} \subseteq \mathcal{Y}_n$, when $n$ is large enough.
 Let $\mathcal{U}$ be the fibre product \[\mathcal{U} = \Delta_G \times_{W} \spec(\FF_q) = \Delta_G \times_{\Delta_{\ZmZ}}  \spec(\FF_q). \]
 There is a morphism $\mathcal{U} \rightarrow \Delta_{\ZpZ}$, and so we define $\mathcal{U}(\leq r)$ to be the preimage of $\Delta_{\ZpZ}(\leq r)$ under this morphism. 
 It then follows from the structure of $\Delta_{\ZpZ}$ and $\Delta_{\ZpZ}(\leq r)$ from \cite[Thm.~$4.13$]{ToniniYasuda2020I} that $\mathcal{U}(\leq r) \subseteq \mathcal{U}$ is the ind-perfection of a DM stack of finite type defined over $\FF_q$.
 Consider the finite \'etale morphism $\mathcal{U} \rightarrow \mathcal{Y}$.
 The image of $\mathcal{U}(\leq r )$ under this morphism is given by $\Delta_{G,\gamma}(\leq r)$.
 Thus it follows that $\Delta_{G,\gamma}(\leq r)$ is defined over $\FF_q$, and satisfies the same properties as $\mathcal{U}(\leq r)$.
  It remains to count the number of $G$-extensions of $\FF_q((t))$ with fixed ramification jump by calculating $\# \Delta_{G,\gamma}(\leq r)(\FF_q) - \# \Delta_{G,\gamma}(\leq r-1)(\FF_q)$.
Suppose $\gamma \neq \gamma^q$ in $\ZmZ$ and consider the fibre product diagram \[
\begin{tikzcd}
\Delta_G \times_{\Delta_{\ZmZ}} W \arrow{r} \arrow{d} & \Delta_G \arrow{d} \\
W \arrow{r} & \Delta_{\ZmZ}
 \end{tikzcd},
\]
where $W$ the connected component of $\Delta_{\ZmZ}$ corresponding to $\gamma$, and $\Delta_G \times_{\Delta_{\ZmZ}} W$ is identified with $\Delta_{G,\gamma}$.
By \cite[Theorem.~$B$]{ToniniYasuda2020I}, $\Delta_{\ZmZ} \otimes_{\FF_q} \overline{\FF}_q$ has $m$ components given by the disjoint union $B\ZmZ \sqcup \cdots \sqcup B\ZmZ$ and these correspond to elements of $\ZmZ$.
The Galois action on the components over $\overline{\FF}_q $ sends $\gamma \mapsto \gamma^q$, and this has fixed points if and only if $\gamma = \gamma^q$ in $\ZmZ$.
 In particular, if $\gamma \neq \gamma^q$, this action has no fixed points, and thus there are no $\FF_q$ points.
 So now we assume that $\gamma = \gamma^q$.
 Let $M$ be the coarse moduli space of $\Delta_{G,\gamma}(\leq r)$, and $M_{\overline{\FF}_q} = M \otimes_{\FF_q}\overline{\FF}_q$. 
 It follows from Lemma \ref{lem; correspondence} that $M_{\overline{\FF}_q}$ is isomorphic to the ind-perfection of $\AAA_{\overline{\FF}_q}^{d+1}/\mu_{m/n^{\dagger}}$, where $d = \dim \Delta_{G,\gamma}(r) - 1$.
By the Lefschetz trace formula,
\[\# M(\FF_q) = \sum_i (-1)^i \text{Tr}(\Frob|H^i_c(M_{\overline{\FF}_q},\QQ_{\ell}))\] where $H^i_c(-,\QQ_{\ell})$ is the compactly supported \'etale $\ell$-adic cohomology.
%
%
The quotient variety $\AAA_{\overline{\FF}_q}^{d+1}/\mu_{m/n^{\dagger}}$ has the same compactly supported $\ell$-adic cohomology groups as $\AAA_{\overline{\FF}_q}^{d+1}$, and we have $H^i_c(\AAA_{\overline{\FF}_q}^{d},\QQ_{\ell}) =  \QQ_{\ell}(-d)$ if $i = 2d$ and it is zero otherwise.
Hence  \[\# M(\FF_q) =\sum_i (-1)^i \text{Tr}(\Frob|H^i_c(M_{\overline{\FF}_q},\QQ_{\ell})) = \sum_i (-1)^i \text{Tr}(\Frob|H^i_c(\AAA_{\overline{\FF}_q}^{d+1},\QQ_{\ell}) ) =  q^{d+1}. \]
Moreover, $\# \Delta_{G,\gamma}(\leq r)(\FF_q) = \frac{1}{|\aut(G)|} \#M(\FF_q)$ where $\aut(G) \cong Z(G)^{-1}$, and so \[\# \Delta_{G,\gamma}(r)(\FF_q)  = \# \Delta_{G,\gamma}(\leq r)(\FF_q) - \# \Delta_{G,\gamma}(\leq r -1)(\FF_q) = |Z(G)|(q-1)q^{d}. \qedhere\]

\end{proof} 
\section{$\mathbf{v}$-functions of representations}\label{sec; v functions}
\subsection{$\mathbf{v}$-functions for indecomposable representations $V_{d,s}$}

In this section we find an explicit formula for the $\mathbf{v}$-function associated to an indecomposable $d$-dimensional representation $V_{d,s}$ of $G$.
Suppose we have a $G$-\'etale $K$-algebra $M$. 
Let $\mathcal{O}_{M}$ be the integral closure of $\mathcal{O}_K$ in $M$.
The direct sum $\mathcal{O}_{M}^{\oplus d}$ has two $G$-actions, the first being the diagonal action coming from the $G$-action on $\mathcal{O}_{M}$ and the second coming from the extension $G \rightarrow GL_{d}(k) \xrightarrow{} GL_{d}(\mathcal{O}_{M})$ of the $G$-action $G \rightarrow GL_{d}(k)$. 
\begin{defi}
    The \emph{tuning module} of $M$ is the submodule $\Xi_M$ of $\mathcal{O}_{M}^{\oplus d}$ on which the two actions coincide.
\end{defi}
\begin{lem}
    The tuning module $\Xi_M$ is a free $k[[t]]$-module of rank $d$.
\end{lem}
\begin{proof}
    See \cite[Prop.~$6.3$]{towardmotivicint}.
\end{proof}
\noindent The tuning module is used to define the $\mathbf{v}$-function.
\begin{defi}
    Let $M/K$ be a $G$-\'etale $K$-algebra and $L/K$ a $G$-extension. 
    Then the $\mathbf{v}$-function of $M$ is given by the formula 
    \[\mathbf{v}(M) = \frac{1}{\# G} \leng_{\mathcal{O}_K}\frac{\mathcal{O}_{M}^{\oplus d}}{\mathcal{O}_{M} \cdot \Xi_M}.\] 
    Moreover, if 
    $(x_{ij})_{1 \leq j \leq d} \in \mathcal{O}_{M}^{\oplus d}$
    is a $k[[t]]$-basis of $\Xi_M$ and $v_{L}$ is the normalised valuation on $L$
    then the $\mathbf{v}$-function is given by 
    \[\mathbf{v}(M) = \frac{1}{\# G}v_L(\det(x_{ij})_{i,j}).\]
\end{defi} \noindent We begin by calculating the $\mathbf{v}$-function for the connected $G$-torsors, corresponding to those $G$-\'etale $K$-algebras that are field extensions $L/K$.
\begin{nota} We write 
$I_{\gamma,r,s} = \{(i,j) \in \{0,1,\cdots,m-1\} \times \{0,1,\cdots, d-1\}: \: i-rj \equiv s \gamma^{-1} \bmod{m}\}$.    
\end{nota}
\begin{lem} \label{v fn lem}
    We have 
    \[\frac{1}{\# G }\sum_{j = 0}^{d-1}v_{L}(\delta^{j}(\alpha^i\beta^j)) = \sum_{(i,j) \in I_{\gamma,r,s}}\frac{i}{m}.\]
\end{lem}
\begin{proof}
    We will show that for all $j \in \{0,\cdots, d-1\}$ we have that $\delta^{j}(\alpha^i\beta^j) = u \alpha^i$ for $u \in k^{\times}$.
    Let $(i,j) \in I_{\gamma,r,s}$. 
    When $j = 0$ we have $\delta^{0}(\alpha^i\beta^{0}) = \alpha^i$, and when $j = 1$ we have 
    \[\delta^{1}(\alpha^i\beta^{1}) = \alpha^i(\beta + 1) - \alpha^i \beta = \alpha^i. \] 
    Suppose we have $\delta^{n}(\alpha^i\beta^{n}) = u \alpha^i$ for some $n \in \{2,\cdots, d-1 \}$ and $u \in k^{\times}$. 
    Then by the proof of Proposition \ref{prop:basis for L} we have
    \[\delta^{n+1}(\alpha^i\beta^{n+1}) = \delta^{n} \left(  \alpha^i\sum_{\ell =1}^{n+1} \binom{n+1}{\ell} \beta^{n+1-\ell} \right),\]
    and by induction this equals $u \alpha^i$ for $u \in k^{\times}$.
    Then it follows that 
    \[v_{L}(\delta^{j}(\alpha^i\beta^j)) = v_{L}(u \alpha^i) = v_{L}(\alpha^i).\] 
    Since 
    $\alpha = t^{\frac{1}{m}}$, 
    we have
    \[\frac{1}{\# G}\sum_{j = 0}^{d-1} v_{L}(\delta^{j}(\alpha^i\beta^j)) =  \frac{1}{\# G} \sum_{(i,j) \in I_{\gamma,r,s}}v_{L}(t^{\frac{i}{m}})  =  \sum_{(i,j) \in I_{\gamma,r,s}} \frac{i}{m}.
    \] \qedhere
\end{proof}

\begin{thm}\label{defn of v function}
    Let $G= \ZpZ \rtimes \ZmZ$, $\gamma \in (\ZmZ)^{\times}$  and $\Delta_{G,\gamma}$ be the space of (connected) $G$-torsors. The $\mathbf{v}$-function associated to an indecomposable $d$-dimensional representation $V_{d,s}$ of $G$ is the function $\mathbf{v}_{V_{d,s}, \gamma}:\Delta_{G,\gamma} \rightarrow \frac{1}{|G|}\ZZ$  given by \[\mathbf{v}_{V_{d,s}, \gamma}(x) = \begin{cases}
        0 & x \: \text{is trivial} \\
        \sum_{(i,j) \in I_{\gamma,r,s}}\frac{i}{m} + \sum_{(i,j) \in I_{\gamma,r,s}} \left\lceil \frac{-iv_{L}(\alpha) - jv_{L}(\beta)}{e_{L/K}} \right\rceil & \text{otherwise}.
    \end{cases}\]
 \end{thm} 
\begin{proof}We begin by finding a basis for the tuning module. 
In our case $\Xi_L$ is given by
 \[\Xi_L = \left\{ (c_1,c_2,\cdots, c_d) \in \mathcal{O}_{L}^{\oplus d}: \begin{array}{ll}
     & \sigma(c_i) = c_{i} +c_{i+1} \: (i<d), \: \sigma(c_d) = c_d \\
     & \tau(c_1) = \xi_1c_1
\end{array} \right\}.\] 
We may rewrite this using the operator 
$\delta = \sigma - \id$. 
We have 
$\delta(c_1) = c_2$, $\delta^2(c_1) = c_3,\cdots, \delta^{d-1}(c_1) = c_d$ 
and 
$\delta^d(c_1) = 0$.
In particular, the tuning module is
\[\Xi_L = \left\{(c, \delta(c), \delta^2(c), \cdots, \delta^{d-1}(c)) \in \mathcal{O}_L^{\oplus d}: \begin{array}{ll}
     &  c \in \mathcal{O}_L^{\delta^d=0} \\
     & \tau(c) = \xi_1c
\end{array} \right\}.\] 
Notice that since we are assuming $\gamma \in (\ZmZ)^{\times}$ as we are in the connected case,
the action of $\tau$ gives 
\[\tau(\alpha^i \beta^j) = \zeta_m^{\gamma i}\alpha^i \zeta_m^{-\gamma rj}\beta^{j} = \zeta_m^{\gamma i - \gamma rj}\alpha^i \beta^{j}.\]
We can use the basis for $\mathcal{O}_L^{\delta^d=0}$ from Corollary \ref{basis for OL} to obtain the following $\mathcal{O}_L$-basis for $\Xi_L$:
\[\bigoplus_{(i,j) \in I_{\gamma,r,s}}k[[t]]\cdot \alpha^i \beta^j t^{n_{i,j}} \] 
where
\[ n_{i,j} = \left\lceil \frac{-iv_{L}(\alpha) - jv_{L}(\beta)}{e_{L/K}} \right\rceil,\] 
and $e_{L/K}$ is the ramification index of $L/K$.
Let $Q$ be the $d \times d$ matrix 
\[\left( \delta^{j}(\alpha^i\beta^j t^{n_{i,j}})_{(i,j) \in I_{\gamma,r,s}}\right).\] 
By Proposition \ref{prop:basis for L}, 
this is upper triangular and its determinant is
\[\det(Q) = \prod_{j = 0}^{d-1}\delta^j(\alpha^i\beta^j t^{n_{i,j}}) = t^{\sum_{(i,j) \in I_{\gamma,r,s}}n_{i,j}}\prod_{j = 0}^{d-1}\delta^j(\alpha^i\beta^j).\] 
Then 
\begin{align*}
        \mathbf{v}_{V_{d,s}, \gamma}(x) 
        = & \frac{1}{\# G}v_{L}(\det(Q)) 
        = \frac{1}{\# G} v_{L}\left( t^{\sum_{(i,j) \in I_{\gamma,r,s}}n_{i,j}} \prod_{j = 0}^{d-1}\delta^{j}(\alpha^i\beta^j)\right) \\
        = & \frac{1}{\# G} v_{L}\left( t^{\sum_{(i,j) \in I_{\gamma,r,s}}n_{i,j}} \right) + \frac{1}{\# G}\sum_{j = 0}^{d-1} v_{L}\left(\delta^{j}(\alpha^i\beta^j) \right).
\end{align*}
The result follows from the fact that 
\[\frac{1}{\# G} v_{L}\left( t^{\sum_{(i,j) \in I_{\gamma,r,s}}n_{i,j}} \right) = v_{K}\left( t^{\sum_{(i,j) \in I_{\gamma,r,s}}n_{i,j}} \right) = \sum_{(i,j) \in I_{\gamma,r,s}}n_{i,j}\]
and the equality from Lemma \ref{v fn lem}.
\end{proof} 
 \noindent By Section \ref{moduli space sec} the infinite-dimensional space $\Delta_{G,\gamma}$ admits a stratification 
\[\Delta_{G,\gamma} = \bigsqcup_{r \in  \ZZ_{\gamma}}\Delta_{G,\gamma}(r),\]
 where each $\Delta_{G,\gamma}(r)$ is the locus of points with ramification jump $r$. 
 The $\mathbf{v}$-function is constant on each $\Delta_{G,\gamma}(r)$, so for every $x \in \Delta_{G,\gamma}(r) \subset \Delta_{G,\gamma} \backslash\{0\}$, we may write $\mathbf{v}_{V_{d,s}, \gamma}(x) = \mathbf{v}_{V_{d,s}, \gamma}(r)$, where \[\mathbf{v}_{V_{d,s}, \gamma}(r) = \sum_{(i,j) \in I_{\gamma,r,s}}\frac{i}{m} +\sum_{(i,j) \in I_{\gamma,r,s}} \left\lceil \frac{-ip + jr}{mp} \right\rceil. \] We will write \[h(r) =  \sum_{(i,j) \in I_{\gamma,r,s}} \left\lceil \frac{-ip + jr}{mp} \right\rceil .\] 
    \begin{lem}\label{lem: change of var v fn} Let $n$ be a non-negative integer, and $D_{V_{d,s}}$ be the numerical invariant $D_{V_{d,s}} = \frac{d(d-1)}{2}$. The $\mathbf{v}$-function satisfies the change of variables formula \[\mathbf{v}_{V_{d,s}, \gamma}(r + n \cdot mp) = \mathbf{v}_{V_{d,s}, \gamma}(r) + n \cdot D_{V_{d,s}}.\]
 \end{lem}

 \begin{proof}
     Since $\mathbf{v}_{V_{d,s}, \gamma}(r) = \sum_{(i,j) \in I_{\gamma,r,s}}\frac{i}{m} + h(r)$, it is enough to consider $h(r + n\cdot mp)$. We have \begin{align*}
         h(r + n\cdot mp) = &  \sum_{(i,j) \in I_{\gamma,r,s}}\left\lceil \frac{-ip - j(r + n\cdot mp)}{mp} \right\rceil \\
         = &\sum_{(i,j) \in I_{\gamma,r,s}} \left( \left\lceil \frac{-ip + jr}{mp} \right\rceil + n \cdot j \right) \\
         = & \left( \sum_{(i,j) \in I_{\gamma,r,s}} \left\lceil \frac{-ip + jr}{mp} \right\rceil \right) + n \cdot  \frac{d(d-1)}{2} = h(r) + n\cdot D_{V_{d,s}}.
     \end{align*}
 \end{proof} 
\subsection{$\mathbf{v}$-functions for non-connected torsors}\label{subsec; v fn non connected}
 
 Let $M = F^{\oplus n}$ be a $G$-\'etale $K$-algebra, where $F/K$ is a $H$-extension of $K$ for a subgroup $H \subset G$.
 By \cite[Lem.~$3.4$]{WoodYasuda2015MassI}, the $\mathbf{v}$-function is convertible, and hence satisfies that 
 \[\mathbf{v}_{V_{d,s}}(M) = \mathbf{v}_{V_{d,s}|_{H}}(F).\] 
 We will also need to consider the $\mathbf{v}$-function when it is restricted to the non-connected torsors.
 The space $\Delta_G^{\tame}$ corresponds to the union of sets of algebras
 \[D^{\tame} = \bigcup_{q \mid m}\{L \cong Q^{\oplus\frac{mp}{q}}: Q/K \: \text{a $\ZZ/q\ZZ$-extension}\}\] 
 where the union is over all divisors $q$ of $m$.
 Each $\ZZ/q\ZZ$ has order prime to $p$, and hence when we restrict to $\ZZ/q\ZZ$ using the convertability property, we find ourselves in the tame case.
 \begin{defi}\label{defi; age fn}
    Let $h \in GL_d(K)$ be an element of order $l$, where $l$ is coprime to $p$, and $\zeta_l$ a primitive $l$-th root of unity. Its eigenvalues may be written as $\zeta_l^{a_1}, \cdots, \zeta_l^{a_d}$, where $0 \leq a_i < l$ for each $1 \leq i \leq d$. The \emph{age} of $h$ is defined as
    \[\age(h) = \frac{1}{l}\sum_{i=1}^{d}a_i.\]
\end{defi} 
\noindent It follows from the convertability property that if $A \in D^{\tame}$, then 
\[\mathbf{v}_{V_{d,s}}(A) = \mathbf{v}_{V_{d,s}}|_{\ZZ/q\ZZ}(Q) = \age(h),\] 
where $q$ is the divisor of $m$ such that $A \cong Q^{\oplus\frac{mp}{q}}$ and $Q$ is a $\ZZ/q\ZZ$-extension.
Here $h \in \ZZ/q\ZZ$ acts on a uniformiser $\pi$ of $Q$ by 
\[h(\pi) = \zeta_e\pi,\] 
where $e$ is the ramification index of $Q/K$. 
It remains to calculate the $\mathbf{v}$-function for the fully ramified non-connected torsors.
For $\gamma \in \ZmZ$ let $m_{\gamma} = \frac{m}{\gcd(\gamma, m)}$.
As in Section \ref{sec; moduli space structure}, we may reduce to the connected case, by considering the set of algebras \[D_{\gamma} = \{L \cong E^{\gcd(\gamma, m)}: E/K \: \text{a $\ZpZ \rtimes \ZZ/m_{\gamma}\ZZ$-extension} \}.\] 
Recall that $\gamma^{\dagger} \in ( \ZZ/m_{\gamma}\ZZ)^{\times}$ comes from the action of $\tau_{\gamma}$ on the generators of the $\ZpZ \rtimes \ZZ/m_{\gamma}\ZZ$-extension.
 \begin{prop}\label{prop; non-connected torsors}
     Let $\gamma \in \ZmZ$ and $H_{\gamma} = \ZpZ \rtimes \ZZ/m_{\gamma}\ZZ$. The $\mathbf{v}$-function $
     \mathbf{v}_{V_{d,s},\gamma}: \Delta_{G, \gamma} \rightarrow \frac{1}{|H_{\gamma}|}\ZZ$ 
     is given by the formula
     \[\mathbf{v}_{V_{d,s},\gamma}(x) = \begin{cases}
        0 & x \: \text{is trivial} \\
         \sum_{(i,j) \in I_{\gamma,r,s}}\frac{i}{m_{\gamma}} + \sum_{(i,j) \in I_{\gamma,r,s}} \left\lceil \frac{-iv_{L}(\alpha) - jv_{L}(\beta)}{m_{\gamma}p} \right\rceil & \text{otherwise},
    \end{cases}\] where \[I_{\gamma,r,s} = \{(i,j) \in \{0,1,\cdots,m_{\gamma}-1\} \times \{0,1,\cdots, d-1\}: \: i-rj \equiv s \gamma^{\dagger -1} \pmod{m_{\gamma}}\}.\] 
 \end{prop}

 \begin{proof}
       Let $L \in D_{\gamma}$. Then $L \cong E^{\gcd(\gamma, m)}$ for $E/K$ such that $\Gal(E/K) \cong \ZpZ \rtimes \ZZ/m_{\gamma}\ZZ := H_{\gamma}$. For each $\gamma$, using the convertability property, we may the reduce to the case of connected $H_{\gamma}$-torsors, which gives for non-trivial $x$, \[\mathbf{v}_{V_{d,s}}|_{H_{\gamma}}(x) = \sum_{(i,j) \in I_{\gamma,r,s}}\frac{i}{m_{\gamma}} + \sum_{(i,j) \in I_{\gamma,r,s}} \left\lceil \frac{-iv_{L}(\alpha) - jv_{L}(\beta)}{m_{\gamma}p} \right\rceil .\]
 \end{proof}
\noindent When $\gcd(\gamma, m) = 1$, this formula is reduced to the case of connected $G$-torsors, as $m_{\gamma} = m$ in this case
\noindent When considering the torsors that are purely wildly ramified, we have the following corollary. \begin{cor}\label{cor; wild v funct}
     The $\mathbf{v}$-function $\mathbf{v}_{V}: \Delta_{G,\gamma = 0} \rightarrow \frac{1}{p}\ZZ$ is given by the formula \[\mathbf{v}_{V_{d,s}}(x) = \begin{cases}
        0 & x \: \text{is trivial} \\
         \sum_{j=0}^{d-1} \left\lceil \frac{ -jv_{L}(\beta)}{p} \right\rceil & \text{otherwise}.
    \end{cases}\]
 \end{cor}

 \begin{proof}
     As the $\ZmZ$-part is unramified, we have $v_L(\alpha) = 0$. When $\gamma = 0$, $\gcd(m,\gamma) = m$, and hence $m_{\gamma} = 1$. Then the tame part satisfies $\sum_{(i,j) \in I_{\gamma,r,s}}\frac{i}{m_{\gamma}} = \sum_{(i,j) \in I_{\gamma,r,s}} i =0$.
 \end{proof} 

\subsection{Decomposable representations}

From now on, denote by $V$ a $d$-dimensional $G$-representation.
Suppose $V$ has a decomposition into decomposables given by 
\[V = \bigoplus_{\lambda = 1}^{l}V_{d_{\lambda },s_{\lambda}} \quad (1 \leq d_{\lambda} \leq p), \sum_{\lambda = 1}^ld_{\lambda},\] 
where each $V_{d_{\lambda },s_{\lambda}}$ is indecomposable of dimension $d_{\lambda}$. 
In this case the numerical invariant $D_V$ is given by \[D_V = \sum_{\lambda = 1}^l \frac{d_{\lambda}(d_{\lambda} - 1)}{2}.\] 
By \cite[Lem.~$3.4$]{WoodYasuda2015MassI}, the $\mathbf{v}$-function is additive with respect to direct sums.
Therefore for a decomposable representation $V$ the $\mathbf{v}$-function is given by the following formula.

\begin{align*}
    \mathbf{v}_V(r)
    = & \sum_{\lambda = 1}^l \mathbf{v}_{V_{d_{\lambda },s_{\lambda}}}(r) \\
    = &  \sum_{\lambda = 1}^l \left(  \sum_{(i,j) \in I_{\gamma,r,s_{\lambda}}}\frac{i}{m_{\gamma}} +\sum_{(i,j) \in I_{\gamma,r,s_{\lambda}}} \left\lceil \frac{-ip + jr}{m_{\gamma}p} \right\rceil \right).
\end{align*}
 
\section{Stringy motives}\label{sec; Stringy motives} \noindent In this section, we calculate the stringy motive $M_{\str}(X)$ of the quotient variety $X = V/G$. 

\subsection{Motivic integration}

\begin{defi}\label{def; grothendieck ring}
    The Grothendieck ring of varieties over a field $k$, denoted by $K_0(\Var_k)$, is an abelian group generated by isomorphism classes of varieties $[Y]$ satisfying that if $Z$ is a closed subvariety of $Y$ then we have that $[Y] = [Y \backslash Z] [Z]$. The ring structure comes from the operation \[ [Y][Z] = [Y \times Z]. \] We set $\LL = [\AAA_k^1]$.
\end{defi} \noindent For our purposes, we modify $K_0(\Var_k)$ by quotienting out by the following relation: Let $f : Y \rightarrow Z$ be a morphism of varieties and $m \geq 0$ be an integer. Then if every geometric fiber of $f$ over an algebraically closed field $L$ is universally homeomorphic to the quotient of $\AAA_L^m$ by a finite linear group action, then we have $[Y] = \LL^m[Z]$. We denote this modified ring of Grothendieck varieties by $K_0'(\Var_k)$. We write $\mathcal{M}' = K_0'(\Var_k)[\LL^{-1}]$ for the localisation of $ K_0'(\Var_k)$ by $\LL$. There is a filtration of $\mathcal{M}'$ by subgroups \[F_m  = \langle [X]\LL^i : | \dim X + i \leq m \rangle, \]and we define \[\hatM = \lim_{\leftarrow}\mathcal{M}'/F_m. \] This is a commutative ring that is complete with the induced topology.

\begin{defi}
    An \emph{$n$-jet of a variety $X$} is a morphism \[\spec k[[t]]/(t^{n+1})\rightarrow X.\] The \emph{$n$-jet scheme} of $X$, $J_nX$, is the fine moduli space of $n$-jets of $X$, and so hence for any ring $A$ we have \[(J_nX)(A) = \Hom(\spec A[[t]]/(t^{n+1}),X).\] The \emph{arc space of X} is the projective limit \[J_{\infty}X = \lim_{n \rightarrow \infty }J_nX.\]
\end{defi}
\noindent For $n \in \NN$, let $\pi_n:J_{\infty }X \rightarrow J_nX$ be the truncation map from the arc space to the $n$-jet scheme.
\begin{defi}
    $C \subseteq J_{\infty}X$ is a stable subset if: there exists some $n$ such that $\pi_n(C) \subseteq J_nX$ is a constructible subset, we have that $C = \pi_n^{-1}(\pi_n(C))$ and the map $\pi_{m+1}(C) \rightarrow \pi_m(C)$ is a piecewise trivial $\AAA_k^d$ bundle for all $m \geq n$.
\end{defi}

\noindent The measure of a stable subset $C \subseteq J_nX$ is defined to be \[\mu_X(C) = [\pi_n(C)] \LL^{-nd} \in \hatM.\] For a general measurable subset $C \subseteq J_nX$, its measure is the limit of measures of stable subsets. Let $C \subseteq J_nX$ be a measurable subset and let $F: C  \rightarrow \ZZ \cup \{ \infty \}$ be a function. $F$ is a \emph{measurable function} if every fibre of $F$ is measurable.
For a measurable function $F$, we may define the integral 
\[ \int_C \LL^F = \sum_{a \in \ZZ}\mu_X(F^{-1}(a))\LL^a \in \hatM \cup \{\infty \}.\] 
We wish to calculate the integral $M_{\str}(X) =  \int_{\Delta_G} \mathbb{L}^{d-\mathbf{v}_{V}}$.
For constructible subsets $C \subseteq \Delta_{G, \gamma}(r) $ the measure $\nu(C)$ is defined to be 
\[\nu (C) = [C] \in \hat{\mathcal{M}}'.\]

\begin{thm}\label{stringy motive thm}
    Let $X = V/G$ be a quotient variety where $V$ is a $d$-dimensional $G$-representation and for every $\gamma \in \ZmZ$, let $\mathbf{v}_{V,\gamma}$ be the $\mathbf{v}$-function associated to $V$. Then if $D_{V} \geq p$ we have 
    \[ M_{\str}(X)  = \sum_{g \in \ZmZ}\LL^{d - \age(g)}  
         + (\LL - 1) \LL^{d - 1}\left(\frac{\sum_{\gamma^{-1} = 0}^{m-1}\sum_{\substack{s = 1\\s \in \ZZ_{\gamma}}}^{m_{\gamma}p - 1}\LL^{\dim \Delta_{G,\gamma}(s)-\mathbf{v}_{V,\gamma}(s)}}{1-\LL^{p - 1 - D_{V}}} \right).\]
\end{thm}
\begin{proof}
Let $X = V/G$ and $\mathbf{v}_{V}$ be the $\mathbf{v}$-function associated to $V$.
We wish to compute \[M_{\str}(X) =  \int_{\Delta_G} \mathbb{L}^{d-\mathbf{v}_{V}}.\]
We separately calculate the contributions from the tame and wild torsors using the decomposition 
\[\Delta_G =  \Delta_G^{\tame} \sqcup \Delta_G^{\wild} = \Delta_G^{\tame} \sqcup \left( \bigsqcup_{\gamma^{-1} = 0}^{m-1}  \Delta_{G,\gamma} \right).\] 
We first consider $\int_{\Delta_G^{\tame}} \mathbb{L}^{d-\mathbf{v}_{V}}$.
The space of tame $G$-torsors corresponds to the set of $G$-\'etale $K$-algebras 
\[D^{\tame}=\bigcup_{q \mid m}\{L \cong Q^{\oplus\frac{mp}{q}}: Q/K \: \text{a $\ZZ/q\ZZ$-extension}\}.\]
For any two $g_1,g_2\in \ZZ/q\ZZ$, we define an equivalence relation $\sim$ via conjugation by elements in $G$.
Specifically, there exists $h \in G$ such that $g_1 = hg_2h^{-1}$. 
For each $q$, we write
\[D_q^{\tame} = \{L \cong Q^{\oplus\frac{mp}{q}}: Q/K \: \text{a $\ZZ/q\ZZ$-extension}\},\] 
and for all $L \in D_q^{\tame}$, $\mathbf{v}_{V}(L) = \age(g)$ for a generator $g$ of $\ZZ/q\ZZ$.
Let $\conj^{\tame}(G)$ be the set of conjugacy classes of $G$ of order coprime to $p$.
Then \begin{align*}
        \int_{\Delta_G^{\tame}} \mathbb{L}^{d-\mathbf{v}_{V}} = & \sum_{A \in D^{\tame}}\LL^{d -\mathbf{v}_{V}(A)} = \sum_{q \mid m}\sum_{A \in D_q^{\tame}}\LL^{d -\mathbf{v}_{V}(A)} \\
        = & \sum_{q \mid m} \sum_{[g] \in (\ZZ/q\ZZ) / \sim }\LL^{d - \age(g)} = \sum_{[g] \in \conj^{\tame} (G)}\LL^{d - \age(g)}.
    \end{align*} \noindent However, we have that since the order of $\ZmZ$ is coprime to $[G:\ZmZ] = p$ by assumption, it is a Hall subgroup of $G$ and in particular, $\ZpZ$ is a normal complement of $\ZmZ$ in $G$. By \cite[Thm.~$1$]{Suzuki1963}, it follows that there are exactly $m$ tame conjugacy classes $\conj^{\tame}(G)$, as $\ZmZ$ is abelian and each element is in its own conjugacy classes. Fixing an embedding $\ZmZ \hookrightarrow G$, we may identify $\conj^{\tame}(G)$ with $\ZmZ$, and hence the sum $\sum_{[g] \in \conj^{\tame} (G)}\LL^{d - \age(g)}$ may be be written as a sum over $\ZmZ$.
For the wild torsors, 
since the spaces $\Delta_{G,\gamma}$ are infinite dimensional for each $\gamma$,
we use the stratification by ramification jump to calculate the integral.
Recall that we use the quotient $(\mathbb{G}_m \times \AAA_{k}^{\dim \Delta_{G,\gamma}(r) - 1})/\mu_{m/n^{\dagger}}$ as a parameter space of $\Delta_{G,\gamma}(r)$. 
However, by Lemma \ref{lem; correspondence}, the class in $\hatM$ of the quotient of an affine space by a linear finite group action is the same as the class of the affine space itself.
Hence in our calculation of the motivic integral it is enough to consider $\nu (\mathbb{G}_m \times \AAA_k^{\dim \Delta_{G,\gamma}(r) -1})$.
We have
\begin{align*}
    \int_{\bigsqcup_{\gamma^{-1} = 0}^{m-1}  \Delta_{G,\gamma}} \mathbb{L}^{d-\mathbf{v}_{V,\gamma}}
    = & \sum_{\gamma^{-1} = 0}^{m-1}  \int_{\Delta_{G,\gamma}} \mathbb{L}^{d-\mathbf{v}_{V,\gamma}} \\
    = & \sum_{\gamma^{-1} = 0}^{m-1} \sum_{r \in \ZZ_{\gamma}}[\mathbb{G}_m \times \AAA_k^{\dim \Delta_{G,\gamma}(r) -1}]  \LL^{d -\mathbf{v}_{V,\gamma}(r)} \\
    = &(\LL - 1) \LL^{d - 1} \sum_{\gamma^{-1} = 0}^{m-1} \sum_{r \in \ZZ_{\gamma}} \LL^{\dim \Delta_{G,\gamma}(r) - \mathbf{v}_{V,\gamma}(r)}.
\end{align*} 
Let 
$m_{\gamma} = \frac{m}{\gcd(\gamma,m)}$.
Then 
\[\sum_{\gamma^{-1} = 0}^{m-1} \sum_{r \in \ZZ_{\gamma}} \LL^{\dim \Delta_{G,\gamma}(r) - \mathbf{v}_{V,\gamma}(r)} = \sum_{\gamma^{-1} = 0}^{m-1}\sum_{r \in \ZZ_{\gamma}}\LL^{\left\lfloor \frac{r - 1}{m_{\gamma}} \right\rfloor + 1 - \left\lfloor \frac{\left\lfloor \frac{r - 1}{m_{\gamma}} \right\rfloor + 1}{p}\right\rfloor -\mathbf{v}_{V,\gamma}(r) }. \]
For any integer $n$ and $1 \leq s \leq (m_{\gamma}p)-1$, 
write $r = (m_{\gamma}p)\cdot n + s$. 
Then applying a change of variables, 
it follows by Lemmas \ref{lem: change of var dim} and \ref{lem: change of var v fn} that this is equal to
\[\sum_{\gamma^{-1} = 0}^{m-1} \sum_{\substack{s = 1\\s \in \ZZ_{\gamma}}}^{m_{\gamma}p - 1} \LL^{\dim \Delta_{G,\gamma}(s) - \mathbf{v}_{V,\gamma}(s)}\sum_{n = 0}^{\infty}\LL^{(p - 1 - D_{V})n}.\] 
This converges if $D_{V} \geq p$, and we obtain that the contribution from the wild part is 
\[(\LL - 1) \LL^{d - 1}\left(\frac{\sum_{\gamma^{-1} = 0}^{m-1}\sum_{\substack{s = 1\\s \in \ZZ_{\gamma}}}^{m_{\gamma}p - 1}\LL^{\dim \Delta_{G,\gamma}(s)-\mathbf{v}_{V,\gamma}(s)}}{1-\LL^{p - 1 - D_{V}}} \right).\] 
\end{proof}

\begin{cor}
    If $D_{V} = p$ then we have
    \[M_{\str}(X) = \sum_{g \in \ZmZ}\LL^{d - \age(g)} + \LL^d \sum_{\gamma^{-1} = 0}^{m-1}\sum_{\substack{s = 1\\s \in \ZZ_{\gamma}}}^{m_{\gamma}p - 1}\LL^{\dim \Delta_{G,\gamma}(s)-\mathbf{v}_{V,\gamma}(s)} \in \ZZ[\LL] .\]
\end{cor}

\noindent The formula for $M_{\str}(X)$ when $D_{V} \geq p$ is a rational function in $\LL$.
We define the \emph{stringy Euler number} $e_{\str}(X)$ of $X$ to be the rational number obtained by substituting $1$ for $\LL$ in the formula for $M_{\str}(X)$.
\begin{prop}\label{prop;str euler number}
    Let $X = V/G$ be the quotient variety from Theorem \ref{stringy motive thm}. 
    Then we have
    \[ e_{\str}(X) = \frac{m D_V}{D_{V} - p + 1}.\]
\end{prop}\noindent Suppose we have a crepant resolution $f:Y \rightarrow X$.
Then we have that 
\[M_{\str}(X) = [Y].\]
In particular, since 
\[e_{\topo}(Y) = e_{\topo}(M_{\str}(X)) = e_{\str}(X),\]
we have obtained a formula for the $\ell$-adic Euler characteristic of a crepant resolution $f:Y \rightarrow X$.
One should note that the formula for $e_{\str}(X)$ and hence $e_{\topo}(Y)$ for a crepant resolution $Y$ (provided it exists) may be written in the form \[\# \conj^{\tame}(G) + \frac{m(p-1)}{D_{V} - p + 1},\] from which it is clear that in general, Batyrev's formula $e_{\topo}(Y) = \# \conj(G)$ does not hold in the wild case.
\subsection{Stringy motives for representations of $S_3$}\label{subsec; S_3 examples stringy motive}
We now calculate the stringy motive for two quotient varieties with known crepant resolutions, coming from two different representations of $S_3$.
Let $P_2$ be the $3$-dimensional representation coming from the projective cover of the sign representation.
By \cite[Thm.~$4.6$]{yamamoto2021crepant}, there exists a crepant resolution $Y \rightarrow P_2/S_3$.
We use Theorem \ref{stringy motive thm} to calculate the stringy motive \[M_{\str}(P_2/S_3) = \int_{\Delta_{S_3}} \mathbb{L}^{3-\mathbf{v}_{P_2}}.\]
We may decompose $\Delta_{S_3}$ as \[\Delta_{S_3} = \Delta_{S_3}^{\tame} \sqcup \Delta_{S_3}^{\wild}. \] 
The integral over $\Delta_{S_3}^{\tame} $ is given by \[\int_{\Delta_{S_3}^{\tame}} \mathbb{L}^{3-\mathbf{v}_{P_2}} = \sum_{[g] \in \conj^{\tame}(S_3)}\LL^{3-\age(g)}.\]
There are two conjugacy classes in $\conj^{\tame}(S_3)$, 
the trivial class corresponding to the unramified torsor contributing $\LL^{3}$ to the total integral, 
and the non-trivial class, which satisfies $\age(g) = 1$, 
and hence this contributes $\LL^2$.
Thus we have 
\[\int_{\Delta_{S_3}^{\tame}} \mathbb{L}^{3-\mathbf{v}_{P_2}} = \LL^3+\LL^2.\]
For the wild torsors, 
we have $m_{\gamma}=1$ when $\gamma^{-1} = 0$ and $m_{\gamma} = 2$ when $\gamma^{-1} = 1$.
The integral over $\Delta_{S_3}^{\wild}$ is given by
\begin{align*}
  \int_{\Delta_{S_3}^{\wild}} \mathbb{L}^{3-\mathbf{v}_{P_2}}
  = &  (\LL - 1) \LL^{2} \left(\frac{\sum_{\gamma^{-1} = 0}^{1}\sum_{\substack{s = 1\\s \in \ZZ_{\gamma}}}^{m_{\gamma}p - 1}\LL^{\dim \Delta_{S_3,\gamma}(s)-\mathbf{v}_{P_2,\gamma}(s)}}{1-\LL^{-1}} \right) \\
    = & \LL^3 \left(\sum_{\substack{s = 1\\s \in \ZZ_{0}}}^{2}\LL^{\dim \Delta_{S_3,0}(s)-\mathbf{v}_{P_2,0}(s)}+\sum_{\substack{s = 1\\s \in \ZZ_{1}}}^{5}\LL^{\dim \Delta_{S_3,1}(s)-\mathbf{v}_{P_2,1}(s)} \right).
\end{align*} 
For $\gamma^{-1} = 0$ and $\gamma^{-1}=1$,
the sets $\ZZ_{\gamma}$ are
$\ZZ_0 = \{s \in \ZZ: 3\nmid s>0\} = \{1,2\}$ and
$\ZZ_1 = \{e \in \ZZ : 3 \nmid s>0, s \equiv 1 \bmod{2}\} = \{1,5\}$. 
The $\mathbf{v}$-function is given by the formula
\[\mathbf{v}_{P_2,\gamma}(r) = \begin{cases}
    \left\lceil \frac{r}{3} \right\rceil + \left\lceil \frac{2r}{3}  \right\rceil, & \gamma^{-1} = 0,\\
    1 + \left\lceil \frac{r}{6} \right\rceil + \left\lceil \frac{2r - 3}{6}  \right\rceil, & \gamma^{-1} = 1.
\end{cases}\]
Then applying the formula for the dimension from Corollary \ref{cor;non connected dim}, 
we obtain that
\begin{align*}
     \int_{\Delta_{S_3}^{\wild}} \mathbb{L}^{3-\mathbf{v}_{P_2}} = \LL^3(\LL^{-1} + \LL^{-1} + \LL^{-1}+\LL^{-2}) = 3\LL^2 + \LL.
\end{align*} 
Thus, the formula for stringy motive is 
\[M_{\str}(P_2/S_3) = \LL^3 + 4\LL^2 + \LL\]
and taking $\LL \rightarrow 1$,
we see that the stringy Euler number $e_{\str}(X)$ equals $6$.
In particular, the $\ell$-adic Euler characteristic of the crepant resolution $Y$ is $6$.
Another known example of a crepant resolution in the wild case comes from a $6$-dimensional representation of $S_3$. 
Let $P_1$ be the $3$-dimensional indecomposable representation coming from the projective cover of the trivial representation.
There is a morphism 
\[\text{Hilb}^3[\AAA_K^{2}] \rightarrow P_1 \oplus P_1/S_3\] 
where $\text{Hilb}^3[\AAA_K^{2}]$ is the Hilbert scheme of three points of $\AAA^{2}_K$, known as the Hilbert-Chow morphism.
It is known from \cite{Beauville1983, KumarThomsen2001, brion2005frobenius} that this morphism is crepant.
We calculate the stringy motive 
\[M_{\str}(P_1 \oplus P_1/S_3) = \int_{\Delta_{S_3}} \mathbb{L}^{6-\mathbf{v}_{P_1 \oplus P_1}}.\] 
The tame torsors contribute
\[\int_{\Delta_{S_3}^{\tame}} \mathbb{L}^{6-\mathbf{v}_{P_1 \oplus P_1}} =\sum_{[g] \in \conj^{\tame}(S_3)}\LL^{6-\age(g)} = \LL^6 + \LL^5,\] 
as $\age(g) = 0$ for the trivial conjugacy class and $\age(g) = 1$ for the non-trivial conjugacy class. 
For the wild torsors we have \begin{align*}
    \int_{\Delta_{S_3}^{\wild}}& \mathbb{L}^{6-\mathbf{v}_{P_1 \oplus P_1}}
    = (\LL - 1)\LL^5\left(\frac{\sum_{\gamma^{-1}=0}^{1}\sum_{\substack{s = 1\\s \in \ZZ_{\gamma}}}^{m_{\gamma}p - 1}\LL^{\dim \Delta_{S_3,\gamma}(s)-\mathbf{v}_{P_1 \oplus P_1,\gamma}(s)}}{1 - \LL^{-4}} \right) \\
    = & (\LL - 1)\LL^5\left(\frac{\sum_{\substack{s = 1\\s \in \ZZ_{0}}}^{2}\LL^{\dim \Delta_{S_3,0}(s)-\mathbf{v}_{P_1 \oplus P_1,0}(s)}+\sum_{\substack{s = 1\\s \in \ZZ_{1}}}^{5}\LL^{\dim \Delta_{S_3,1}(s)-\mathbf{v}_{P_1 \oplus P_1,1}(s)}}{1 - \LL^{-4}}\right)
\end{align*}
Using the additivity property of the $\mathbf{v}$-function, we have \[\mathbf{v}_{P_1 \oplus P_1} = 2\mathbf{v}_{P_1}.\] 
The $\mathbf{v}$-function satisfies 
\[\mathbf{v}_{P_1 \oplus P_1}(r)= \begin{cases}
     2 \left(\left\lceil \frac{r}{3} \right\rceil + \left\lceil \frac{2r}{3}  \right\rceil \right) & \gamma^{-1} = 0,\\
     1 + 2\left( \left\lceil \frac{2r}{6} \right\rceil + \left\lceil \frac{r-3}{6}  \right\rceil  \right) & \gamma^{-1} = 1.
\end{cases}\] 
Therefore the formula gives 
\begin{align*}
    \int_{\Delta_{S_3}^{\wild}} \mathbb{L}^{6-\mathbf{v}_{P_1 \oplus P_1}} 
    = & (\LL - 1)\LL^5\left(\frac{\LL^{-2}+\LL^{-3} +\LL^{-4}+\LL^{-5}}{1 - \LL^{-4}} \right) \\
    = & (\LL - 1)\left( \frac{\LL^7 + \LL^6 +\LL^5 + \LL^4}{(\LL - 1)(\LL + 1)(\LL^2+1)} \right)\\
    = & \LL^4\left( \frac{\LL^3+\LL^2 + \LL + 1}{(\LL + 1)(\LL^2+1)} \right) = \LL^4.
\end{align*} We have obtained the formula for the stringy motive
\[M_{\str}(P_1 \oplus P_1/S_3) = \LL^6 + \LL^5 + \LL^4.\]
We note that this matches the motivic version of Bhargava's mass formula in \cite[Cor.~$16.5$]{Yasuda2024WildDM}. 
In this case, taking $\LL \mapsto 1$ gives stringy Euler number $e_{\str}(P_1 \oplus P_1/S_3) =3$, and thus we obtain that the $\ell$-adic Euler characteristic of $\text{Hilb}^3[\AAA_K^{2}]$ is $3$.
In particular, 
it is different to when we take the $\mathbf{v}$-function coming from the $3$-dimensional representation.

\section{The $a$- and $b$- invariants}\label{sec; a and b invs}

\subsection{Raising functions}
In this section, we will assume that the base field $k$ is finite.
We recall the following definition from \cite[Def.~$4.1$]{darda2025batyrevmaninconjecturedmstacks}.

\begin{defi}
    A raising function $c: |\Delta_G| \rightarrow\RR_{\leq 0}$ is a constructible function satisfying the following properties:
    \begin{enumerate}
        \item One has that $c(x) = 0$ if and only if $x$ is the trivial element.
        \item The function $c$ is invariant for the isomorphisms of $|\Delta_G|$ which are induced from an automorphism $\FF_{q'}((t)) \xrightarrow{\sim} \FF_{q'}((t))$.
    \end{enumerate}
\end{defi} \noindent By \cite[Thm.~$4.9$]{moduliformaltorsorsii}, $\mathbf{v}$-functions are examples of raising functions. A raising function has the following associated invariants:

\begin{defi}
    The $a$-invariant associated to a raising function $c$ is given by \[a(c) = \sup_{g \in c(\Delta_G) \backslash \{0\}} \frac{1 + \dim(c^{-1}(g))}{g}.\]
    Let $D(c)$ be the set 
    \[D(c) = \{g \in c(\Delta_G) \backslash \{0\} : 1 + \dim (c^{-1}(g)) = a(c)g\}.\] 
    The $b$-invariant associated to $c$ is given by 
    \[ b(c) = \sum_{g \in D(c)}\# \left\{ \begin{array}{ll}
         & \text{irreducible components of $c^{-1}(g)$ of} \\
         & \text{dimension equal to $\dim(c^{-1}(g))$}
    \end{array}\right\}.\]
\end{defi}

\noindent In our case, the $a$-invariant depends on the value of the $\mathbf{v}$-function at both the wild and tame torsors.

\begin{lem}\label{lem; formula for a inv}
    Let $G = \ZpZ \rtimes \ZmZ$ and $\mathbf{v}_{V}$ be the $\mathbf{v}$-function associated to the $d$-dimensional $G$-representation $V$. 
    Then \begin{equation}
    a(\mathbf{v}) = \max \left(\max_{[g] \in \conj^{\tame}(G)\backslash\{ [1]\} }(\age(g))^{-1}, 
    \max_{r \in \cup_{\gamma^{-1}=0}^{m-1}\ZZ_{\gamma}}\frac{1 + \dim \Delta_{G,\gamma}(r)}{\mathbf{v}_{V,\gamma}(r)} \right).
\end{equation}

\end{lem}

\begin{proof}
As the decomposition of $\Delta_G$ from $\S$\ref{moduli space sec} holds over a finite field, we may write \[\Delta_G = \Delta_{G}^{\tame} \sqcup \Delta_{G}^{\wild}.\] 
    Then
    \begin{align*}
         \sup_{g \in \mathbf{v}_{V}(\Delta_{G}^{\tame} \sqcup\Delta_{G}^{\wild})\backslash \{0\}}&\frac{1+\dim(\mathbf{v}_{V}(g)^{-1})}{g} \\
         = & \sup \left(\sup_{g \in \mathbf{v}_{V}(\Delta_{G}^{\tame})\backslash \{0\}}\frac{1+\dim(\mathbf{v}_{V}(g)^{-1})}{g}, \sup_{g \in \mathbf{v}_{V}(\Delta_{G}^{\wild})\backslash \{0\}}\frac{1+\dim(\mathbf{v}_{V}(g)^{-1})}{g} \right).
    \end{align*}
   When restricted to $\Delta_{G}^{\tame}$, the $\mathbf{v}$-function is equal to the age function, and the space is zero-dimensional. Hence we have that \[\sup_{g \in \mathbf{v}_{V}(\Delta_{G}^{\tame})\backslash \{0\}}\frac{1+\dim(\mathbf{v}_{V}(g)^{-1})}{g} = \max_{[g] \in \conj^{\tame}(G)\backslash\{ [1]\} }(\age(g))^{-1}.\] For the wild part, we use the stratification by ramification jump \[\Delta_{G}^{\wild} = \bigsqcup_{\gamma^{-1} = 0}^{m-1} \bigsqcup_{r \in \ZZ_{\gamma}}  \Delta_{G,\gamma}(r).\] Then for $g = \mathbf{v}_{V,\gamma}(r)$ we have \[\sup_{g \in \mathbf{v}_{V}(\Delta_{G}^{\wild})\backslash \{0\}}\frac{1+\dim(\mathbf{v}_{V}(g)^{-1})}{g}  =   \max_{r \in \cup_{\gamma^{-1}=0}^{m-1}\ZZ_{\gamma}}\frac{1 + \dim \Delta_{G,\gamma}(r)}{\mathbf{v}_{V,\gamma}(r)}.\] \qedhere
\end{proof}

\noindent We may use the $a$-invariant to study the discrepancies of singularities of $X = \AAA^d_k/G$.
Let $f: X \rightarrow Y$ be a modification (a proper birational morphism) and $K_{Y/X} = K_{Y} - f^*K_X = \sum_i a_iE_i$ be the relative canonical divisor, where the $a_i \in \QQ$ and the $E_i$ are the irreducible components.
We call the $a_i$ the \emph{discrepancies} of the $E_i$ with respect to $X$. The minimal discrepancy is \[\delta (X) = \inf_f \min_i a_i.\] 
The quotient variety $X$ is \emph{terminal} (resp. \emph{canonical}, \emph{log terminal}, \emph{log canonical}) if $\delta(X) > 0$ (resp. $\geq 0$, $> -1$, $\geq -1$). If $\delta(X) < -1$ then it is equal to $- \infty $. 
We have that \begin{align*}
    \dim \int_{\Delta_G \backslash \{0\}}\LL^{-\mathbf{v}_{V}} 
    = & \sup_g(\dim(\mathbf{v}_{V}^{-1}(g)) - g) \\
    = & d - \min \log \text{discrep}(\AAA_k^d / G).
\end{align*}
Then the minimal log discrepancy of $\AAA_k^d / G$
is zero if and only if $\dim \int_{\Delta_G \backslash \{0\}}\LL^{-\mathbf{v}_{V}} = -1$. 
For every $g > 0$,
we have the equivalence \[\sup_{g \in \mathbf{v}_{V}(\Delta_{G})}\frac{1+\dim(\mathbf{v}_{V}(g)^{-1})}{g} \leq 1 \iff \sup_g(\dim(\mathbf{v}_{V}^{-1}(g)) - g) \leq -1.\] 
In particular,
the value of the $a$-invariant provides information on the dimension $\dim \int_{\Delta_G \backslash \{0\}}\LL^{-\mathbf{v}_{V}}$,
and in turn, this provides information on the singularities of $X$.
\begin{thm}\label{thm; quotient singularities}
    Let $G = \ZpZ\rtimes\ZmZ$ and $X = \AAA_k^{d}/G$ be the quotient variety. $X$ has canonical (resp. terminal) singularities if and only if $\age(g) \geq 1$ (resp. $> 1$) for all $[g] \in \conj^{\tame}(G)\backslash \{[1]\}$ and $\dim \Delta_{G,\gamma}(r) - \mathbf{v}_{V,\gamma }(r) \leq -1$ (resp. $< -1$) for all $r \in \ZZ_{\gamma}$ in the range $\{1,\cdots m_{\gamma}p -1\}$.
\end{thm}
\begin{proof}
   We prove the case of canonical singularities as the case of terminal singularities is analogous. We have that $X$ has canonical singularities if and only if the $a$-invariant $a(\mathbf{v}) \leq 1$. In particular, we need the maximum of $\max_{[g] \in \conj^{\tame}(G)\backslash\{ [1]\} }(\age(g))^{-1}$ and $\max_{r \in \cup_{\gamma^{-1}=0}^{m-1}\ZZ_{\gamma}}\frac{1 + \dim \Delta_{G,\gamma}(r)}{\mathbf{v}_{V,\gamma}(r)}$ to be less than or equal to $1$. This happens if and only if $\age(g) \geq 1$ for all $[g] \in \conj^{\tame}(G)\backslash \{[1]\}$ and $\dim \Delta_{G,\gamma}(r) - \mathbf{v}_{V,\gamma }(r) \leq -1$ for all $r \in \ZZ_{\gamma}$ in the range $\{1,\cdots m_{\gamma}p -1\}$.
\end{proof}
\noindent In contrast to the tame case, it is possible for $X$ to not have any canonical singularities despite all non-trivial elements $g \in \conj^{\tame}(G)$ having $\age(g) \geq 1$. Therefore, both the tame and wild torsors contribute to the existence of singularities. 
If a quotient variety $\AAA_k^d / G$ has a crepant resolution, this implies that it only has canonical singularities and in particular, that the $a$-invariant is equal to $1$.
Furthermore, in this case, the $b$-invariant is the number of exceptional prime divisors on the crepant resolution.
It may be read off the formula for $M_{\str}(\AAA_k^d / G)$ as the coefficient of the $\LL^{d-1}$ term.
\subsection{$a$- and $b$- invariants for $S_3$}
We explicitly compute the $a$- and $b$-invariants of the $\mathbf{v}$-functions coming from the representations $P_2$ and $P_1 \oplus P_1$ of $S_3$ and show that the $a$-invariant is indeed equal to $1$.
As the decomposition of $\Delta_{S_3}$ does not change when working over a finite field, the computation from $\S$\ref{subsec; S_3 examples stringy motive} remain valid in this case.
The $a$-invariant is given by the formula
\[a(\mathbf{v}_{P_2}) = \sup_{g \in \mathbf{v}_{P_2}(\Delta_{S_3}) \backslash \{0\}} \frac{1 + \dim(\mathbf{v}_{P_2}^{-1}(g))}{g}.\]
We consider the supremum taken over all non-trivial 
$g \in \mathbf{v}_{P_2}( \Delta_{S_3}^{\tame} ), \mathbf{v}_{P_2}(\Delta_{S_3}^{\wild} )$. 
For $g \in \mathbf{v}_{P_2}( \Delta_{S_3}^{\tame})$, we let $g = \age(h)$. 
Then 
\[\sup_{g \in \mathbf{v}_{P_2}( \Delta_{S_3}^{\tame}) \backslash \{0\}} \frac{1 + \dim(\mathbf{v}_{P_2}^{-1}(g))}{g} = \max_{g \in \conj^{\tame}(S_3)\backslash \{0\}} (\age(g))^{-1}=1,\]
and this comes from one irreducible component.
For
$g \in \mathbf{v}_{P_2}(\Delta_{S_3}^{\wild})$, 
we consider
\[ \sup_{g \in \mathbf{v}_{V_{d,s},\gamma}(\Delta_{S_3}^{\wild})}\frac{1+\dim(\mathbf{v}_{V_{d,s},\gamma}(g)^{-1})}{g} = \max_{r \in \ZZ_{0}\cup \ZZ_{1}}\frac{1 + \dim \Delta_{S_3,\gamma}(r)}{\mathbf{v}_{P_2,\gamma}(r)}.\] 
Since we only need consider the 
$r \in \{1,\cdots, m_{\gamma}p - 1\}$, 
the relevant $r \in \ZZ_0$ are $r =1,2$ and the relevant $r \in \ZZ_{1}$ are $r = 1,5$. 
Then we have
\[
    \max_{r \in \{1,2\}}\frac{1 + \dim \Delta_{S_3,\gamma = 0}(r) }{\mathbf{v}_{P_2,\gamma = 0}(r)} = 1
\]
and this is reached for both irreducible components $r = 1$ and $r=2$.
For $\gamma = 1$ we have \[ \max_{r \in \{1,5\}}\frac{1 + \dim \Delta_{S_3,\gamma = 1}(r) }{\mathbf{v}_{P_2,\gamma = 1}(r)} = 1\]
and this is reached at one irreducible component corresponding to $r =1$. So we have $a(\mathbf{v}_{P_2}) = 1$ and $b(\mathbf{v}_{P_2}) = 4$.
In particular, we have proved the following.
\begin{prop}
    The crepant resolution $f: Y \rightarrow P_2/S_3$ has four exceptional prime divisors.
\end{prop}
\noindent On the other hand, we consider the $a$- and $b$-invariants for the $6$-dimensional representation.
For $g \in \mathbf{v}_{P_1 \oplus P_1}( \Delta_{S_3}^{\tame})$, 
we once again have that
\[\max_{g \in \conj^{\tame}(S_3)\backslash \{0\}} (\age(g))^{-1}=1.\] 
If $g  \in \mathbf{v}_{P_1 \oplus P_1}(\Delta_{S_3}^{\wild})$,
we have that \[\max_{r \in \ZZ_{0}\cup \ZZ_{1}}\frac{1 + \dim \Delta_{S_3,\gamma}(r)}{\mathbf{v}_{P_1 \oplus P_1,\gamma}(r)} < 1\] for all $r \in \ZZ_0 \cup \ZZ_1$.
In this case we have that $a(\mathbf{v}_{P_1 \oplus P_1}) = 1$ and $b(\mathbf{v}_{P_1 \oplus P_1}) = 1$.
We see that both representations of $S_3$ have $a$-invariant equal to $1$. It follows that both $P_2/S_3$ and $P_1 \oplus P_1/S_3$ have canonical singularities, but no terminal singularities.
\bibliographystyle{amsalpha}
\bibliography{ms}
\end{document}